\documentclass{article}
\usepackage[margin=1in]{geometry}
\usepackage{epsfig,amsmath}
\usepackage{graphicx}
\usepackage{amssymb}
\usepackage{enumerate}
\usepackage{comment}
\usepackage{color}
\usepackage{paralist}

\numberwithin{equation}{section}

\newtheorem{lemma}{Lemma}[section]

\newtheorem{propn}[lemma]{Proposition}

\newtheorem{thm}[lemma]{Theorem}
\newtheorem{cor}[lemma]{Corollary}
\newtheorem{defn}[lemma]{Definition}
\newtheorem{remark0}[lemma]{Remark}

\DeclareMathOperator{\SLE}{SLE}

\DeclareMathOperator{\GFF}{GFF}

\newenvironment{proof}{{\em Proof.}}{\hspace*{\fill} $\square$ \medskip}

\newenvironment{proof2}[1]{{\em Proof of #1.}}{\hspace*{\fill} $\square$}

\begin{document}
	\newcommand{\test}{test}
	
	\newcommand{\R}{\mathbb{R}}
	\newcommand{\RR}{\mathbb{R}}
	\newcommand{\C}{\mathbb{C}}
	\newcommand{\N}{\mathbb{N}}
	\newcommand{\Z}{\mathbb{Z}}
	\newcommand{\I}{\mathbf{1}}
	\newcommand{\E}[2]{\mathbb{E}_{#1}\left[ #2 \right]}
	\newcommand{\HH}{\mathbb{H}}
	\newcommand{\dd}{\,{\mathrm d}}
	\newcommand{\ddd}{{\mathrm d}}
	\newcommand{\st}{\,\partial}
	\newcommand{\stt}{\partial}
	\newcommand{\e}{\operatorname{e}}
	\newcommand{\im}{\mathrm{i}}
	\newcommand{\eps}{\varepsilon}
	\newcommand{\grad}{\operatorname{grad}}
	\newcommand{\divg}{\operatorname{div}}
	\newcommand{\Prb}[2]{\mathbb{P}_{#1} \left({#2} \right)}
	\newcommand{\Prbs}[2]{\mathbf{P}_{#1} \left({#2} \right)}
	\newcommand{\Prbn}[1]{\mathbb{P} \left({#1} \right)}
	\newcommand{\ip}[2]{\left\langle #1, #2 \right\rangle}
	\newcommand{\hl}{-\frac{1}{2}\Delta}
	\newcommand{\hh}{\hat{h}}
	\newcommand{\he}{\hat{\eta}}
	\newcommand{\cc}[2]{#1\Scdot #2^{-1}+Q\log|(#2^{-1})'|}
	\newcommand{\re}[1]{#1_{R,\eps}}
	\newcommand{\tl}[1]{\tilde{#1}}
	\newcommand{\qiq}{q\in \mathcal{Q}}
	\definecolor{dgreen}{rgb}{0.2,0.8,0.6}
	\definecolor{dblue}{rgb}{0.3,0.3,1}
	\newcommand{\ellen}[1]{\textcolor{dgreen}{[#1]}}
	\newcommand{\nina}[1]{\textcolor{cyan}{[#1]}}
	\newcommand{\ec}[1]{\textcolor{dblue}{#1}}
	\newcommand{\nc}[1]{\textcolor{magenta}{#1}}
	\newcommand{\en}[1]{#1}
	\newcommand{\Hloc}{H^{-1}_{\op{loc}}}
	\newcommand{\Hinv}{H^{-1}}
	\newcommand{\B}{\mathbb{B}}
	\newcommand{\D}{\mathbb{D}}
	\newcommand{\Nz}{\mathbb{N}_0}
	\newcommand{\Q}{\mathbb{Q}}
	\renewcommand{\P}{\mathbb{P}}
	\newcommand{\bbH}{\mathbb{H}}
	\newcommand{\ga}{\gamma}
	\renewcommand{\GFF}{\op{GFF}}
	\newcommand{\ep}{\varepsilon}
	\newcommand{\1}{\mathbf{1}}
	\renewcommand{\Re}{\mathrm{Re}}
	\renewcommand{\Im}{\mathrm{Im}}
	\newcommand{\scr}{\mathscr}
	\def\cZ{\mathcal{Z}}
	\def\cY{\mathcal{Y}}
	\def\cX{\mathcal{X}}
	\def\cW{\mathcal{W}}
	\def\cV{\mathcal{V}}
	\def\cU{\mathcal{U}}
	\def\cT{\mathcal{T}}
	\def\cS{\mathcal{S}}
	\def\cR{\mathcal{R}}
	\def\cQ{\mathcal{Q}}
	\def\cP{\mathcal{P}}
	\def\cO{\mathcal{O}}
	\def\cN{\mathcal{N}}
	\def\cM{\mathcal{M}}
	\def\cL{\mathcal{L}}
	\def\cK{\mathcal{K}}
	\def\cJ{\mathcal{J}}
	\def\cI{\mathcal{I}}
	\def\cH{\mathcal{H}}
	\def\cG{\mathcal{G}}
	\def\cF{\mathcal{F}}
	\def\cE{\mathcal{E}}
	\def\cD{\mathcal{D}}
	\def\cC{\mathcal{C}}
	\def\cB{\mathcal{B}}
	\def\cA{\mathcal{A}}
	\def\cl{\mathfrak{l}}
	\def\strip{\mathbf{S}}
	\newcommand{\aryb}{\begin{eqnarray*}}
		\newcommand{\arye}{\end{eqnarray*}}
	\def\alb#1\ale{\begin{align*}#1\end{align*}}
	\newcommand{\eqb}{\begin{equation}}
	\newcommand{\eqe}{\end{equation}}
	\newcommand{\eqbn}{\begin{equation*}}
	\newcommand{\eqen}{\end{equation*}}
	\newcommand{\BB}{\mathbb}
	\newcommand{\ol}{\overline}
	\newcommand{\ul}{\underline}
	\newcommand{\op}{\operatorname}
	\newcommand{\la}{\langle}
	\newcommand{\ra}{\rangle}
	\newcommand{\bd}{\mathbf}
	\newcommand{\frk}{\mathfrak}
	\newcommand{\eqD}{\overset{d}{=}}
	\newcommand{\rtaD}{\overset{d}{\rightarrow}}
	\newcommand{\rta}{\rightarrow}
	\newcommand{\xrta}{\xrightarrow}
	\newcommand{\Rta}{\Rightarrow}
	\newcommand{\hookrta}{\hookrightarrow}
	\newcommand{\wt}{\widetilde}
	\newcommand{\wh}{\widehat} 
	\newcommand{\mcl}{\mathcal}
	\newcommand{\pre}{{\operatorname{pre}}}
	\newcommand{\lrta}{\leftrightarrow}
	\newcommand{\bdy}{\partial}
	\newcommand{\tr}{\op{tr}}
	\newcommand{\neu}{\op{Neu}}
	\newcommand{\uo}{\op{u}}
	\renewcommand{\L}{{\op{L}}}
	\newcommand{\Ro}{{\op{R}}}
	\newcommand{\rad}{{\op{rad}}}
	\newcommand{\cir}{{\op{circ}}}
	\newcommand{\w}{{\op{wedge}}}
	\newcommand{\Bessel}{\cB}

	\title{Conformal welding for critical Liouville quantum gravity} 
	\date{}
	\author{
		\begin{tabular}{c} Nina Holden\\ [-5pt] \small ETH Z\"urich \end{tabular}
		\begin{tabular}{c} \\[-5pt]\small  \end{tabular}
		\begin{tabular}{c} \\[-5pt]\small  \end{tabular}
		\begin{tabular}{c} \\[-5pt]\small  \end{tabular}
		\begin{tabular}{c} Ellen Powell\\ [-5pt] \small ETH Z\"urich \end{tabular}
	}
	\maketitle
	
	\begin{abstract}
		Consider two critical Liouville quantum gravity surfaces (i.e., $\gamma$-LQG for $\gamma=2$), each with the topology of $\HH$ and with infinite boundary length. We prove that there a.s.\ exists a conformal welding of the two surfaces, when the boundaries are identified according to quantum boundary length.  
		This results in a critical LQG surface decorated by an independent $\SLE_4$.
		Combined with the proof of uniqueness for such a welding, recently established by McEnteggart, Miller, and Qian (2018), this shows that the welding operation is well-defined. Our result is a
		critical analogue of Sheffield's quantum gravity zipper theorem (2016), which shows that a similar conformal welding for subcritical LQG (i.e., $\gamma$-LQG for $\gamma\in(0,2)$) is well-defined. \medskip
		
	\end{abstract}

	\medskip

	\section{Introduction} 
	Let $\D_1$ and $\D_2$ be two copies of the unit disk $\D$, and suppose that $\phi:\partial \D_1\to\partial \D_2$ is a homeomorphism. Then $\phi$ provides a way to identify the boundaries of $\D_1$ and $\D_2$, and hence produce a topological sphere. The classical \emph{conformal welding} problem is to endow this topological sphere with a natural conformal structure. When the sphere is uniformised (i.e., when it is conformally mapped to $\BB S^2$) we get a simple loop $\eta$ on $\BB S^2$, which is the image of the unit circle. 
	Equivalently, the conformal welding problem consists of finding a triple $\{\eta,\psi_1,\psi_2\}$, where $\eta$ is a simple loop on $\BB S^2$, and $\psi_1$ and $\psi_2$ are conformal transformations taking $\D_1$ and $\D_2$, respectively, to the two components of $\BB S^2\setminus\eta$, such that $\phi= \psi_2^{-1}\circ \psi_1$. If such a triple exists and is uniquely determined by $\phi$ (up to M\"obius transformations of the sphere) then one says that the conformal welding (associated to $\phi$) is well-defined.  
	
	The extension of this problem to the setting of \emph{random} homeomorphisms has received much attention in recent years; in particular, when the random curves and homeomorphisms are related to natural conformally invariant objects such as Schramm--Loewner evolutions 
	(SLE) and Liouville quantum gravity (LQG). 
	This will be the focus of the present paper. In particular, we consider the case of critical ($\gamma=2$) LQG, which is associated with SLE$_4$.
	
	Roughly speaking, LQG is a theory of random fractal surfaces obtained by distorting the Euclidean metric by the exponential of a real parameter $\gamma$ times a Gaussian free field (GFF). Such random surfaces give rise to random conformal welding problems, for instance, when the homeomorphism $\phi$ corresponds to gluing the boundaries of two discs according to their LQG-boundary lengths. Weldings of this type have been studied in several recent works \cite{ajks10,ajks11,Sh16,DMS18+,MMQ18}. In particular, for a class of homeomorphisms defined in terms  subcritical LQG measures ($\gamma$-LQG for $\gamma\in(0,2)$) existence and uniqueness of the conformal welding was established by Sheffield \cite{Sh16}, and the interface $\eta$ was proven to have the law of an SLE$_{\kappa}$ with $\kappa=\gamma^2\in(0,4)$. Uniqueness of a random conformal welding where the interface $\eta$ has the law of an SLE$_4$ was recently established by McEnteggart, Miller, and Qian \cite{MMQ18}. 
	
	Let us now make the set-up more precise. Given a parameter $\gamma\in(0,2]$, a simply connected domain $D\subset\C$, and an instance $h$ of (some variant of) a GFF on $D$, one would heuristically like to define the $\gamma$-LQG ``surface" associated with $(D,h)$ to be the 2d Riemannian manifold with metric tensor $e^{\gamma h}(dx^2+dy^2)$ on $D$. This definition does not make rigorous sense since $h$ is a distribution and not a function, but one can prove by regularising the field (\cite{Kah85,RV10,DS11,Ber17}) that $h$ induces a so-called ``$\gamma$-LQG area measure" $\mu^\gamma_h$ in $D$ (with formal definition 
	$\e^{\gamma h(z)} dx dy$) and a ``$\gamma$-LQG boundary length measure" $\nu^\gamma_h$ along $\partial D$ (with formal definition 
	$\e^{(\gamma/2) h(x)} ds$). The case $\gamma=2$ is known as \emph{critical}, because the regularisation procedure used when $\gamma\in(0,2)$ breaks down at this point, and defining the critical measure requires a different strategy.
	
	Given two pairs $(D_1,h_1)$ and $(D_2,h_2)$, such that $0<\nu^\gamma_{h_1}(\partial D_1)=\nu^\gamma_{h_2}(\partial D_2)<\infty$ one may define the homeomorphism $\phi$ that identifies $\partial D_1$ and $\partial D_2$ according to these boundary lengths. That is, $\phi:\partial D_1\to\partial D_2$ is such that for all $A\subset\partial D_1$, $\nu^\gamma_{h_1}(A)= \nu^\gamma_{h_2}(\phi(A))$.
	One can then ask if the conformal welding associated to $\phi$, as described above, is well-defined. 
	
	In fact, it is more convenient to consider this problem in the setting where $(D_i,h_i)$ for $i=1,2$ have \emph{infinite} boundary length. To explain the interpretation of the conformal welding problem in this framework, and to state our main theorem, we need the following definition. For a simply connected domain $D\subseteq\C$ let $H^{-1}_{\op{loc}}(D)$ denote the space of generalised functions $h$ on $D$ such that for any open set $U$ with $\overline{U}\Subset D$, the distribution $h|_U$ is in the Sobolev space $H^{-1}(U)$. 
	\begin{defn}
		\label{def::lqg}
		Let $\gamma\in(0,2]$. A \emph{$\gamma$-Liouville quantum gravity (LQG) surface} is an equivalence class of pairs $(D,h)$, where $D\subseteq \C$ is a simply connected domain and $h\in H^{-1}_{\op{loc}}(D)$ is a distribution (or generalised function) on $D$. 
		Two pairs $(D_1,h_1)$ and $(D_2,h_2)$ are defined to be equivalent if there is a conformal map $\psi:D_2\to D_1$ such that  \eqb h_2=h_1\circ\psi+Q_\gamma\log|\psi'| \label{eqn:coc}, \eqe where $Q_\gamma=2/\gamma+\gamma/2$.\footnote{Note that this equivalence relation depends on $\gamma$. Also note that $h_1\in H^{-1}_{\op{loc}}(D_1)$ if and only if $h_2\in H^{-1}_{\op{loc}}(D_2)$.}
	\end{defn}
	
	It follows from the regularisation procedure used to define the LQG measures that if $h_1$ and $h_2$ are related as in \eqref{eqn:coc}, then the push-forward of $\mu_{h_2}^\gamma$ (resp., $\nu_{h_2}^\gamma$) by $\psi$ is equal to $\mu_{h_1}^\gamma$ (resp., $\nu_{h_1}^\gamma$).
	
	In this paper the distribution $h$ will always be a Gaussian free field or a related kind of distribution.
	We think of two equivalent pairs $(D_1,h_1)$ and $(D_2,h_2)$ as two different \emph{parametrisations} of the same $\gamma$-LQG surface; indeed, the previous paragraph implies that they describe equivalent LQG measures. We will often abuse notation and refer to $(D,h)$ as a $\gamma$-LQG surface, i.e., we identify $(D,h)$ with its equivalence class. \en{If we introduce a $\gamma$-LQG surface $\cS$ by writing $\cS=(D,h)$ we mean that $\cS$ is a $\gamma$-LQG surface (i.e., an equivalence class) while $(D,h)$ is a particular parametrisation of this surface.} Recall that by the Riemann mapping theorem, a quantum surface comes equipped with a well-defined notion of topology: either that of $\HH$ (equivalently, some other bounded simply connected domain), $\C$, or $\BB S^2$. 
	
	%When we say that two $\gamma$-LQG surfaces are independent (e.g. in Theorem \ref{thm1} right below) we mean that the equivalence classes are independent.}
	
	We also consider \emph{marked} quantum surfaces; these are tuples $(D,h,z_1,\dots,z_k)$ for $k\in\N$ and $z_1,\dots,z_k\in D\cup \partial D$. In order for two marked quantum surfaces $(D_1,h_1,z_1,\dots,z_k)$ and $(D_2,h_2,w_1,\dots,w_k)$ to be equivalent we require that there exists a $\psi$ as in Definition \ref{def::lqg}, which also satisfies $z_1=\psi( w_1),\dots,z_k=\psi(w_k)$. 
	
	Let us now come back to conformal welding: we will consider the following alternative version of the problem. Suppose that $\HH_1$, $\HH_2$ are two copies of the upper half-plane and $\phi$ is a homeomorphism from $\R_+$ to $\R_-$. The problem is to find a triple $\{\eta, \psi_1,\psi_2\}$, where $\eta$ is a simple curve in $\HH$ from $0$ to $\infty$ and $\psi_1,\psi_2$ are conformal transformations taking $\HH_1$ and $\HH_2$ to the two components of $\HH\setminus \eta$, such that $\phi=\psi_2^{-1}\circ \psi_1$. If such a triple exists and is unique then we say that the conformal welding associated to $\phi$ is well-defined.
	
	Given two doubly-marked $\gamma$-quantum surfaces with the topology of $\HH$, parametrised by $(h_1,\HH,0,\infty)$ and  $(h_2,\HH,0,\infty)$, and such that $\nu^\gamma_{h_1}(\R_+)=\infty$, $\nu^\gamma_{h_2}(\R_-)=\infty$, but $\nu_{h_1}^\gamma,\nu_{h_2}^\gamma$ give finite and positive mass to bounded intervals of positive length, we can define the homeomorphism $\phi$ that identifies $\R_+$ and $\R_-$ according to $\nu_{h_1}^\gamma,\nu_{h_2}^\gamma$ boundary length. That is, $\nu_{h_1}^\gamma([0,a])=\nu_{h_2}^\gamma([\phi(a),0])$ for all $a\in [0,\infty)$. The main result of this paper is that for certain critical $(\gamma=2)$ quantum surfaces known as quantum wedges (see Section \ref{sec::lqgandwedges}), this conformal welding problem has a solution. See Figure \ref{fig:welding} for an illustration.

	\begin{thm}	Let $\cS=(\BB H,h,0,\infty)$ be a $(2,1)$-quantum wedge, and let $\eta$ be an $\SLE_4$ from 0 to $\infty$ which is independent of $h$. Let $D^\L\subset\BB H$ (resp., $D^\Ro\subset\BB H$) be the points of $\BB H$ lying strictly to the left (resp., right) of $\eta$, and define the 2-LQG surfaces $\cS^\L=(D^\L,h|_{D^\L},0,\infty)$ and $\cS^\Ro=(D^\Ro,h|_{D^\Ro},0,\infty)$. 
		
		Then $\cS^\L$ and $\cS^\Ro$ are independent 2-LQG surfaces, and each surface has the law of a $(2,2)$-quantum wedge. Furthermore, the quantum boundary lengths along $\eta$ as defined by $\cS^\L$ and $\cS^\Ro$ agree. \vspace{0.1cm}
		\label{thm1}
	\end{thm}
	
	\begin{figure}[h]
		\centering
		\includegraphics[scale=1]{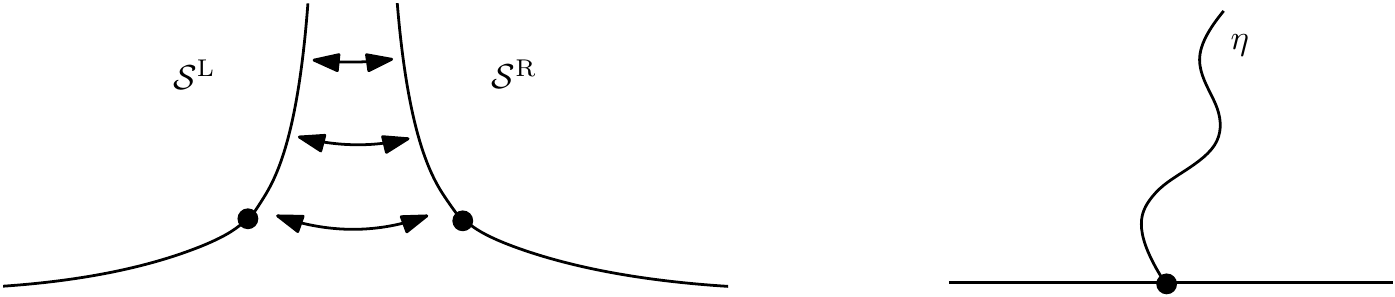}
		\caption{Illustration of the conformal welding problem. We get a topological half-plane by welding together the two surfaces $\cS^{\mathrm{L}}$ and $\cS^{\mathrm{R}}$. By Corollary \ref{cor:wd}, if $\cS^{\mathrm{L}}$ and $\cS^{\mathrm{R}}$ are independent $(2,2)$-quantum wedges and the welding is defined in terms of $2$-LQG boundary length, then the resulting surface (a $(2,1)$-quantum wedge) has an a.s.\ uniquely defined conformal structure, and the interface $\eta$ has the law of an SLE$_4$.} 
		\label{fig:welding}
	\end{figure}
	
	\begin{figure}
		\centering
		\includegraphics[scale=1]{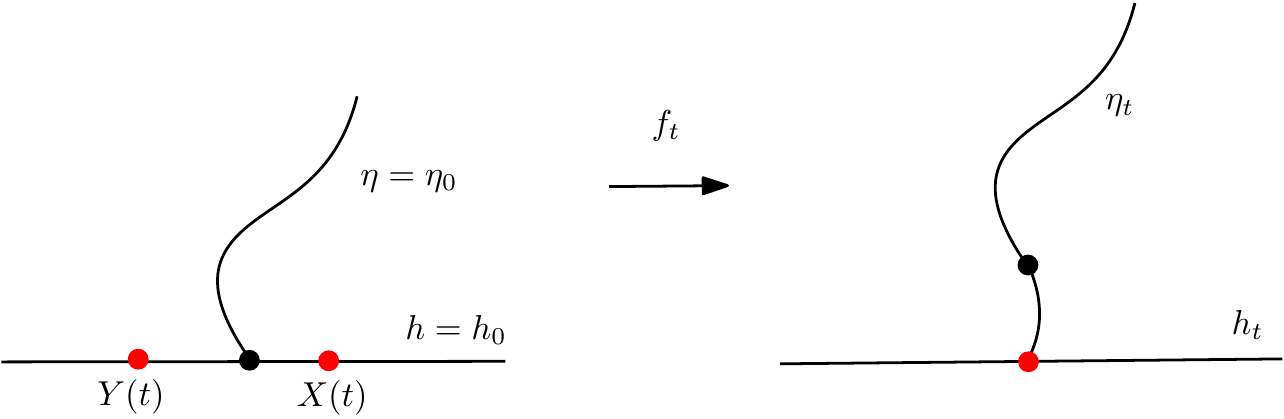}
		\caption{Consider a $(2,1)$-quantum wedge $(\HH,h,0,\infty)$ decorated by an independent SLE$_4$ $\eta$. The quantum zipper identifies segments $[0,X(t)]$ and $[Y(t),0]$, each of quantum length $t>0$. This gives a new surface/curve pair $(h_t,\eta_t)$ with the same law as before. By Theorem \ref{thm:criticalzipper}, the processes of zipping up and zipping down are measurable with respect to $(h,\eta)$.
		}
		\label{fig:zipper}
	\end{figure}
	
	\en{We remark that independence of the $2$-LQG surfaces $\cS^\L$ and $\cS^\Ro$ in Theorem \ref{thm1} does \emph{not} mean that the fields $h|_{D^\L}$ and $h|_{D^\Ro}$ are independent;} \en{these two fields are dependent e.g.\ since they induce the same quantum length measure along $\eta$.} 
	\en{Instead, we have independence of the two surfaces viewed as equivalence classes. This means that if we embed the two surfaces in some standard form then the fields in this embedding are independent. Explicitly, if $(\BB H, \wt h^\L,0,\infty)$ is an embedding of $\cS^\L$ such that (say) the unit half-circle has unit mass, and $(\BB H, \wt h^\Ro,0,\infty)$ is defined similarly for $\cS^\Ro$, then the fields $\wt h^\L$ and $\wt h^\Ro$ are independent.}
	
	\en{By Theorem \ref{thm1} we have a quantum  length measure along $\eta$ which is defined by considering the LQG boundary measure of the surfaces $\cS^\L$ and $\cS^\Ro$. We remark that this length measure along $\eta$ can be defined equivalently in a more intrinsic way by considering $e^{h}d\frk m$, where $\frk m$ is the measure supported on $\eta$ given by its $3/2$-dimensional Minkowski content. This equivalence was proved for the subcritical zipper in \cite{benoist} and the critical case follows by the same argument. 
	}
	
	The following uniqueness result concerning the conformal welding problem of Theorem \ref{thm1} was recently established in \cite[Theorem 2]{MMQ18}.
	\begin{thm}[McEnteggart-Miller-Qian '18]
		Let $\eta$ be an SLE$_4$ in $\HH$ from 0 to $\infty$. Suppose that $\varphi:\HH\to\HH$ is a homeomorphism which is conformal in $\HH\setminus\eta$ and such that $\varphi(\eta)$ has the same law as $\eta$. Then $\varphi$ is a.s.\ a conformal automorphism of $\HH$.
		\label{thm2}
	\end{thm}
	Hence, if $\{\eta, \psi_1,\psi_2\}$ and $\{\eta', {\psi}'_1,{\psi}'_2\}$ are two solutions to the conformal welding problem associated with the same homeomorphism $\phi$, and $\eta$, $\eta'$ both have the law of $\SLE_4$, then applying the above theorem to the map $\varphi$ which is set equal to $ \psi'_2 \circ \psi_2^{-1}$ on the left of $\eta$ and $\psi'_1 \circ \psi_1^{-1}$ on the right, it follows that $\varphi$ must be a conformal automorphism of $\HH$. 
	Theorems \ref{thm1} and \ref{thm2} together therefore imply that the conformal welding operation for critical LQG is well-defined. 
	\begin{cor}\label{cor:wd}
		Consider two $(2,2)$-quantum wedges 
		$\cS^{\op{L}}=(\HH,h^{\op{L}},0,\infty)$ and 
		$\cS^{\op{R}}=(\HH,h^{\op{R}},0,\infty)$, and identify the boundary arc 
		$[0,\infty)$ of $\cS^{\op{L}}$ and the boundary arc $(-\infty,0]$ of $\cS^{\op{R}}$ 
		according to $2$-LQG boundary length. This a.s.\ gives a uniquely defined conformal welding of the two 2-LQG surfaces such that the interface $\eta$ between the surfaces has the law of a chordal SLE$_4$.
	\end{cor}
	Observe that the conformal welding in this corollary is not proven to be the unique conformal welding among \emph{all} possible conformal weldings; since it is assumed in Theorem \ref{thm2} that the curves $\varphi(\eta)$ and $\eta$ both have the law of SLE$_4$ curves, we only obtain uniqueness among the weldings for which the interface has this law. The uniqueness result can be strengthened to curves a.s.\ satisfying certain deterministic geometric properties by using the stronger variant of Theorem \ref{thm2} found in \cite[Theorem 2]{MMQ18}.
	
	\vspace{0.1cm}
	We also obtain a dynamic version of the critical conformal welding, analogous to Sheffield's quantum gravity zipper \cite[Theorem 1.8]{Sh16} in the case $\gamma\in(0,2)$. See Figure \ref{fig:zipper} for an illustration. 
	\begin{thm}
		\label{thm:criticalzipper}
		Let $(\HH,h_0,0,\infty)$ be the equivalence class representative of a $(2,1)$-quantum wedge with the last exit parametrization (see Definition \ref{def:lastexit}).\footnote{The theorem is still true if we let $(\HH,h_0,0,\infty)$ be some other equivalence class representative of a $(2,1)$-quantum wedge, provided the field $h_0$ is measurable with respect to the LQG surface, i.e., the equivalence class representative is chosen in a measurable way relative to the surface.} Let $\eta_0$ be an SLE$_4$ from $0$ to $\infty$ in $\HH$ which is independent of $h_0$.
		Then for every $t>0$ there exists a conformal map $f_t$ defined on $\HH$, 
		which is measurable with respect to $h_0$, such that: 
		\begin{compactitem}
			\item 
			\en{$(h_t,\eta_t)$ has the same law as $(h_0,\eta_0)$, where\footnote{It can be shown that $h_t\in \Hinv(\HH)$ is well-defined independently of its definition on $\eta_t\setminus f_t(\eta_0)$.} $h_t=h_0\circ f_t^{-1}+2\log|(f_t^{-1})'|$ and $\eta_t$ is the union of $f_t(\eta_0))$ and 
				$\HH\setminus f_t(\HH)$};
			\item if $(X(s))_{0\leq s \le t}$ and $(Y(s))_{0\leq s \le t}$ are such that $\nu_{h_0}([0,X(s)])=\nu_{h_0}([Y(s),0])=s$ for every $s\in [0,t]$, then $f_t$ maps $[0,X(t)]$ and $[Y(t),0]$ to the right- and left-hand sides of $\eta_t\setminus f_t(\eta_0)$, respectively, and for every $s\le t$, $X(s)$ and $Y(s)$ are mapped to the same point on $\eta_t\setminus f_t(\eta_0)$.
		\end{compactitem} 
		This gives rise to a bi-infinite process $(h_t,\eta_t)_{t\in \R}$, such that:
		\begin{compactitem}
			\item $(h_t,\eta_t)_{t\in \R}$ is measurable with respect to $(h_{t_0},\eta_{t_0})$ for any $t_0\in \R$; and
			\item $(h_t,\eta_t)_{t\in \R}$ is stationary, i.e., for any $t_0\in\R$ the two processes $(h_{t_0},\eta_{t_0})_{t\in \R}$ and $(h_{t_0+t},\eta_{t_0+t})_{t\in \R}$ are equal in law.
		\end{compactitem}
	\end{thm}
	
	As described in \cite{Sh16} we can think of the operation $(h_0,\eta_0)\mapsto (h_t,\eta_t)$ for $t>0$ as \emph{zipping up} the surfaces $h_0|_{D^\L},h_0|_{D^\Ro}$ to the left and right of $\eta_0$ by $t$ units of quantum boundary length. Similarly, we think of the operation $(h_0,\eta_0)\mapsto (h_t,\eta_t)$ for $t<0$ as \emph{zipping down}.

	\subsection{Related works}
	Conformal weldings related to LQG were first studied in \cite{ajks10,ajks11}, where it was proven that the conformal welding of a subcritical LQG surface to a Euclidean disk according to boundary length is a.s.\ well-defined (see \cite{tecu} for the case of critical LQG). In Sheffield's breakthrough work \cite{Sh16} it is shown that the conformal welding of \emph{two} subcritical LQG surfaces is a.s.\ well-defined, and that the interface is given by an SLE$_\kappa$ curve. More precisely, the following is proved.
	\begin{thm}[Sheffield '16]
		Consider two $(\gamma,\gamma)$-quantum wedges 
		$\cS^{\op{L}}=(\HH,h^{\op{L}},0,\infty)$ and 
		$\cS^{\op{R}}=(\HH,h^{\op{R}},0,\infty)$, with $\gamma\in (0,2)$, and identify the boundary arc 
		$[0,\infty)$ of $\cS^{\op{L}}$ to the boundary arc $(-\infty,0]$ of $\cS^{\op{R}}$ 
		according to $\gamma$-LQG boundary length. This a.s.\ gives a uniquely defined conformal welding of the two $\gamma$-LQG surfaces. In this conformal welding, the interface $\eta$ between the surfaces has the law of a chordal SLE$_{\gamma^2}$, and the combined surface\footnote{\label{footnote:offcurve}That is, the surface parametrised by $(\HH,h,0,\infty)$ where $h$ is set equal to the image (after welding) of $h^{\op{L}}$ on the left of $\eta$ and of $h^{\op{R}}$ on the right of $\eta$. The field $h$ is a well-defined element of $\Hinv(\HH)$ regardless of how $h$ is defined on $\eta$ itself.} has the law of a $(\gamma,\gamma-2/\gamma)$-quantum wedge \en{that is independent of $\eta$.}
		\label{thm3}
	\end{thm}
	The existence part of Theorem \ref{thm3} is established by studying a certain coupling between a GFF and a reverse SLE$_\kappa$, where the law of the GFF is invariant under zipping up and down the SLE$_\kappa$. 
	The uniqueness part follows from \cite{JS00}, where Jones and Smirnov proved that the boundaries of H\"older domains are conformally removable, and \cite{RS05}, where Rohde and Schramm proved that the complement of an SLE$_\kappa$ for $\kappa\in(0,4)$ is a.s.\ a H\"older domain. For an overview of the proof, we recommend the notes \cite{Ber16}. 
	
	\begin{remark0}
		\label{rmk:sczipper}
		The analogue of Theorem \ref{thm:criticalzipper} is also proved in \cite[Theorem 1.8]{Sh16} in the case $\gamma\in(0,2)$. \en{That is, starting with the curve and combined surface described at the end of Theorem \ref{thm3} (let us call them $(h_0,\eta_0)$) we get a bi-infinite stationary process $(h_t,\eta_t)_{t\in \R}$ that is measurable with respect to $(h_0,\eta_0)$.} 
	\end{remark0}
	
	Duplantier, Miller, and Sheffield \cite{DMS18+} have also studied problems closely related to conformal welding. In particular, they proved that if one considers an SLE$_{\kappa}$ $\eta$ on an independent $\gamma$-LQG surface $\cS$, where $\kappa\gamma^2=16$, then $\eta$ is measurable with respect to a pair of so-called \emph{forested wedges}. These wedges are the restrictions of $\cS$ to the components of the complement of $\eta$ -- one consisting of components traced anti-clockwise by $\eta$, and the other consisting of components traced clockwise 
	-- along with topological information (encoded by a pair of L\'evy processes) about how these components are glued together.
	A number of other measurability results concerning welding of general LQG surfaces are established in the same paper. We note however that these measurability results are of a weaker kind than, for example, the result in \cite{Sh16}. For instance, uniqueness of the ``gluing'' of forested wedges described above is only proved under the assumption that the resulting field $h$ and curve $\eta$ have a particular joint law.
	
	As already mentioned, McEnteggart, Miller, and Qian in the recent paper \cite{MMQ18}, have also proved uniqueness of conformal weldings in certain settings. More precisely, they prove that if $\eta$ is a curve in $\HH$ and $\psi:\HH\to\HH$ is a homeomorphism which is conformal on $\HH\setminus\eta$, then $\psi$ is in fact conformal as soon as $\eta$ and $\psi(\eta)$ satisfy certain geometric regularity conditions. These conditions are in particular satisfied a.s.\ if $\eta$ and $\psi(\eta)$ both have the law of an SLE$_\kappa$ for $\kappa\in(0,8)$. Their result is new for $\kappa\in[4,8)$, while it follows from conformal removability for $\kappa\in(0,4)$.

	\subsection{Outline}
	The rest of the article is structured as follows.
	We begin in Section \ref{sec:prelims} by collecting relevant definitions: of the Gaussian free field and its variants; LQG surfaces and their parametrisations; and the specific quantum surfaces known as quantum wedges that will be particularly important in this paper. Here we also describe the construction of \emph{boundary} LQG measures, and discuss some properties of these measures that are needed in what follows. In particular we will make use of a connection between subcritical and critical measures, that is a consequence of \cite{APS18two}. 
	We conclude the preliminaries by briefly introducing Schramm--Loewner evolutions, and proving some basic convergence results that will be useful later on.
	
	Sections \ref{sec:zoom} and \ref{sec:criticalzipperapprox} provide the key ingredients (Propositions \ref{prop:zoom} and \ref{prop:joint-conv}, respectively) for the proofs of Theorems \ref{thm1} and \ref{thm:criticalzipper}. In Section \ref{sec:zoom} it is shown that if one observes a $2$-LQG surface in a small neighbourhood of a critical LQG-measure typical boundary point, then it closely resembles a $(2,2)$-quantum wedge. This gives the critical LQG analogue of \cite[Proposition 1.6]{Sh16}, justifies why the $(2,2)$-quantum wedge is a natural quantum surface (to our knowledge this is the first time that this surface is defined in the literature), and is important to identify the laws and establish independence of the quantum surfaces $\cS^{\L}$ and $\cS^{\Ro}$ in the proof of Theorem \ref{thm1}.

	In Section \ref{sec:criticalzipperapprox} we prove that Sheffield's subcritical quantum gravity zipper (defined for $\gamma\in(0,2)$), has a limit in a strong sense as $\gamma\uparrow 2$. This is shown by proving and combining various convergence results concerning reverse $\SLE_{\kappa=\gamma^2}$ and $\gamma$-LQG measures as $\gamma\uparrow 2$. The proof requires a careful study of quantum wedges and their associated measures in a neighbourhood of the origin, and analysis of the Loewner equation for points on the real line. 
	As a consequence of this section, we obtain Theorem \ref{thm:criticalzipper}.
	Finally, in Section \ref{sec:mainproofs} we show how the main results of the previous sections allow us to deduce Theorem \ref{thm1}. 
	
	It is also worth taking a moment now to discuss why the proof in \cite{Sh16} does not generalise straightforwardly to the critical case. At a very high level, the key difficulties are: (a) 
	\en{lack of first moments for critical LQG-measures}; 
	and (b) non-Gaussian conditioning for the law of the field around ``quantum typical points''. To explain this in more detail, 
	we first need to describe the general strategy of \cite{Sh16} (for a more complete overview, the reader should consult \cite{Sh16} or \cite{Ber16}). As in the present paper, the fundamental object to construct is the \emph{quantum gravity zipper}: a dynamic coupling between a $(\gamma,\gamma-2/\gamma)$-quantum wedge and an $\SLE_{\kappa=\gamma^2}$ analogous to the coupling described in Theorem \ref{thm:criticalzipper}. From this, the analogue of Theorem \ref{thm1} follows fairly easily.
	
	In order to construct the subcritical quantum gravity zipper, Sheffield first describes a different dynamic coupling, this time between an $\SLE_{\kappa}$ and a Neumann GFF plus a log singularity, that he calls the ``capacity zipper''. The existence  of this coupling is straightforward to prove using a martingale argument. From here, roughly speaking, the ``quantum zipper'' can be obtained by ``zooming in'' at the origin of the capacity zipper. One key tool that is made use of (see, for example, \cite[Proposition 1.6]{Sh16}) is a nice description of the field plus a $\gamma$-quantum typical point, \emph{when the field is weighted by $\gamma$-LQG boundary length}. The difficulty with this in the critical case is that, in contrast to the subcritical setting, critical LQG measures assign mass with infinite expectation to finite intervals. Although this issue is actually possible to circumvent for many purposes -- we will do exactly this using a truncation argument in Section \ref{sec:zoom} -- it causes significant problems if we want to say anything precise about the joint law of the curve and the surface in the critical analogue of the capacity zipper, at a time when a critical quantum typical point is ``zipped up'' to the origin. An additional technical difficulty is created by the fact that critical measures need to be defined using a different approximation procedure to subcritical measures (see Section \ref{sec:gmc}). This means that the law of the field around a quantum-typical point is no longer described in terms of its original law via a simple Girsanov shift, and makes it  difficult to describe how the law of the curve changes in the context mentioned above. For example, it is unclear if it will simply add a drift to the reverse SLE driving function, as is the case when $\gamma\in(0,2)$.
	
	Although it may be possible to obtain the results of this paper by adapting the method of \cite{Sh16} in some way, for the sake of avoiding significant additional technicalities we have chosen the approximation approach.
	
	\vspace{0.1cm}
	{\bf Acknowledgements} N.H.\ acknowledges support from Dr.\ Max R\"ossler, the Walter Haefner Foundation, and the ETH Z\"urich Foundation. E.P.\ is supported by the SNF grant $\#$175505. Both authors would like to express their thanks to Juhan Aru, for his valuable input towards the initiation and strategy of this project, and for numerous helpful discussions. \en{They also thank an anonymous referee for his or her careful reading of the paper and for helpful comments.}
	
	\section{Preliminaries}\label{sec:prelims}
	
	\subsection{Gaussian free field}
	Let $D \subset \BB C$ be a domain with harmonically non-trivial boundary, i.e., such that a Brownian motion started at some point in $D$ hits $\partial D$ a.s. \en{Let $C_0^\infty(D)$ denote the space of infinitely differentiable functions on $D$ with compact support.} For $f,g\in C_0^\infty(D)$ 
	define the Dirichlet inner product of $f$ and $g$ by
	\eqbn
	\langle f,g\rangle_\nabla=\int_D \nabla f\cdot \nabla g\,dxdy.
	\eqen
	Let $H_0(D)$ denote the Hilbert space closure (with respect to this inner product) of the subspace of functions $f \in C^\infty_0(D)$ with
	$\|f\|_\nabla:=\langle f,f\rangle_\nabla<\infty$.\footnote{Note that $H_0(D)$ is the Sobolev space which is often denoted by $H^1_0(D)$ or $W^{1,2}_0(D)$ in the literature. Similarly, the space $H(D)$ defined below is the Sobolev space $H^1(D)=W^{1,2}(D)$.} Let $f_1,f_2,\dots$ be a $\langle\cdot,\cdot\rangle_\nabla$-orthonormal basis for $H_0(D)$. The \emph{zero boundary} Gaussian free field (GFF) $h$ is then defined by setting
	\eqb
	h = \sum_{j=1}^{\infty}\alpha_j f_j,\qquad \text{ where } 
	\alpha_1,\alpha_2,\dots \sim \cN(0,1)\text{ are independent}.
	\label{eq:gff}
	\eqe
	The convergence of \eqref{eq:gff} does not hold in $H_0(D)$ itself, \en{but rather in a space of generalised functions. More precisely,  let $H^{-1}(D)$ be the dual space of $H_0(D)$, equipped with the norm} 
	\en{	\eqb\|k\|_{H^{-1}(D)}=\sup_{g\in H_0(D): \|g\|_{\nabla}=1} ( k, g ) \eqe  where 
		we use the notation $(k, \cdot )$ for the action of $k\in H^{-1}(D)$ on $H^1(D)$}. Note that an element of $H^{-1}(D)$ defines a distribution on $D$ with action $f\mapsto (k,f)$, since $C_0^\infty(D)\subset H^1(D)$. 
	We further define the space $\Hloc(D)$ to be the subspace of generalised functions $h$ on $D$ such that for any open set $U$ with $\bar{U}\Subset D$, the restriction of $h$ to $U$ is in the space $H^{-1}(U)$. We say that $h_n\to h$ in $\Hloc(D)$ if and only if $h_n|_U\to h|_U$ in $H^{-1}(D)$ for any such $U$.
	
	\en{Then the series \eqref{eq:gff} converges a.s.\ in $\Hloc(D)$ and the Gaussian free field $h$ is defined as an element of this space a.s. In particular, $h$ is a.s.\ a random distribution; as above, we write $(h,f)$ for the action of $h$ on $f\in C_0^\infty(D)$. We note that when $D$ is bounded, the series actually converges a.s.\ in $H^{-1}(D)$  and so $h$ is a.s.\ an element of this space. }
	
	\en{	Finally, we mention that for $f\in C_0^\infty(D)$, \eqb\mathrm{var}((h,f))=\langle\Delta^{-1}f,\Delta^{-1} f\rangle_\nabla=(f,\Delta^{-1}f)\eqe
		and so $(h,f)$ actually makes sense (as an a.s.\ limit) for any $f$  such that $(f-f_n,\Delta^{-1}(f-f_n))\to 0$ for some sequence $f_n\in C_0^\infty(D)$. When is $D$ is bounded, for instance, this is exactly the set of functions $f$ in $H^{-1}(D)$.}
	
	For any given bounded and measurable $\rho:\partial D\to\RR$ the GFF with \emph{Dirichlet boundary condition $\rho$} is defined to be a random distribution with the law of $h+\ol\rho$, where $\ol\rho$ is the harmonic extension of $\rho$ to the interior of $D$. 
	
	To define a \emph{mixed boundary condition} GFF, assume that $\partial D$ is divided into two boundary arcs $\partial_{\op{D}}$ and $\partial_{\op{F}}$, and that a function $\rho:\partial D\to\R$  satisfying $\rho|_{\partial_{\op{F}}}=0$ is given. Write $\ol\rho$ for the harmonic extension of $\rho$ to $D$ and let $H_{\partial_{\op{D}},\partial_{\op{F}}}(D)$ be the Hilbert space closure of the subspace of functions $f \in C^\infty(D)$ with
	$\|f\|_\nabla<\infty$ and $f|_{\partial_{\op{D}}}=0$. 
	The \emph{mixed boundary GFF with Dirichlet boundary data $\rho$ on $\partial_{\op{D}}$}, is then defined to be a random distribution with the law of $h+\ol\rho$, where $h$ is now defined by \eqref{eq:gff} with $f_1,f_2,\dots$ an orthonormal basis for $H_{\partial_{\op{D}},\partial_{\op{F}}}(D)$.
	
	To define the \emph{free boundary} GFF (equivalently, the \emph{Neumann} GFF), consider the subspace of functions $f \in C^\infty(D)$ with $\|f\|_\nabla<\infty$. Notice that $\langle\cdot, \cdot\rangle_\nabla$ is degenerate on this subspace of functions, in the sense that $\langle f_C,g\rangle_\nabla=0$ for any $g$ if $f_C\equiv C\in \R$. However, $\langle\cdot, \cdot\rangle_\nabla$ defines a positive definite inner product as soon as we quotient the space by identifying functions that differ by an additive constant. Write $H(D)$ for the Hilbert space closure of this quotient space with respect to the inner product $\langle \cdot, \cdot \rangle_\nabla$. The free boundary GFF $h$ is then defined by \eqref{eq:gff}, where $f_1,f_2,\dots$ is now an orthonormal basis for $H(D)$. Again the convergence of the defining sum does not take place in $H(D)$ itself, \en{but in the quotient space of $\Hloc(D)$ under the equivalence relation that identifies elements differing by an additive constant.} 
	We therefore define the free boundary GFF as an element of \en{$\Hloc(D)$, \emph{modulo an additive constant}}, i.e., $h$ and $h+C$ are identified for any $C\in\R$. One may fix the additive constant in various ways, for example by requiring that the average of $h$ over some fixed set is 0.
	
	When $D=\HH$, by \cite[Lemma 4.2]{DMS18+}, $H(\HH)=H_1(\HH) \oplus H_2(\HH)$ is an orthogonal 
	decomposition of $H(\HH)$, where $H_1(\HH)$ is the subspace of functions $f\in H(\HH)$ that are radially symmetric about the origin (considered modulo an additive constant), and $H_2(\HH)$ is the subspace of functions $f\in H(\HH)$ that have average zero on all semi-circles centred at the origin.
	This induces a decomposition of $\Hloc(\HH)$: any $h\in \Hloc(\HH)$ can be uniquely written as $h=h_\rad+h_\cir$, where $\langle h_\rad,f\rangle_\nabla=0$ for any $f\in H_2(\HH)$ and $\langle h_\cir,f\rangle_\nabla=0$ for any $f\in H_1(\HH)$. In the following we will often make a slight abuse of notation and talk about the ``projection" of an element of $\Hloc(\HH)$ onto $H_1(\HH)$ or $H_2(\HH)$: by this we mean the corresponding projection in $\Hloc(\HH)$.
	
	\begin{remark0}
		To reiterate; in the spaces $\{H^{-1}(D), \Hloc(D), H_{\partial_{\op{D}},\partial_{\op{F}}}(D), H_0(D), H_2(\HH)\}$ functions that differ by an additive constant are \emph{not} identified, while in $\{H(D), H_1(\HH)\}$ they \emph{are} identified.
	\end{remark0}
	
	Finally, we mention that if $f\in H(D)$ and $h$ is a Neumann GFF in $D$, then the law of $h+f$ is absolutely continuous with respect to the law of $h$. Indeed by standard theory of Gaussian processes, the Radon--Nikodym derivative of the former with respect to the latter is proportional to $\e^{\langle h, f \rangle_\nabla}$, where $\langle h, f \rangle_\nabla:=\lim_{n\rta\infty}\langle \sum_{j=1}^n \alpha_j f_j, f \rangle_\nabla$. 
	
	\subsection{Quantum wedges} \label{sec::lqgandwedges}
	Recall the definition of a $\gamma$-LQG surface from the introduction (Definition \ref{def::lqg}).
	
	\emph{Quantum wedges} are a particular family of doubly-marked LQG surfaces which were originally introduced in \cite{Sh16} (see also \cite{DMS18+}). We will parametrise these surfaces by $(\HH,h,0,\infty)$ throughout most of the paper, but also sometimes by the strip $\strip=\RR\times[0,\pi]$ with marked points at $\pm\infty$. \en{These will be related by the} \en{conformal map $\phi:\strip\to\HH$ defined by 
		\eqb
		\phi(z)=\exp(-z),
		\label{eqn:striptoH}
		\eqe
		which sends $\infty$ (resp., $-\infty$) to 0 (resp.,  $\infty$).} When we discuss quantum wedges, there will be two parameters of interest. The first parameter $\gamma$ specifies how we are defining equivalence classes of quantum surfaces (i.e., it plays the role of the parameter $\gamma$ in Definition \ref{def::lqg}) 
	and the second parameter $\alpha$ specifies the weight of a logarithmic singularity that we are placing at the origin. We refer to the surface as a $(\gamma,\alpha)$-quantum wedge. In this paper  we will actually only consider $(\gamma,\gamma-2/\gamma)$-quantum wedges and $(\gamma,\gamma)$-quantum wedges for $\gamma\in (0,2]$. The case $\gamma=2$ has not been considered in earlier papers, but the definition from \cite{Sh16,DMS18+} extends in a natural way to this case. Before we state the formal definition of the $(\gamma,\alpha)$-quantum wedge we need to introduce some notation.

	Since a doubly-marked quantum surface actually refers to an \emph{equivalence class}, and since for any $a>0$ the map $z\mapsto az$ defines a conformal map from $(\HH,0,\infty)$ to $(\HH,0,\infty)$, there are several different fields $h$ that describe the same quantum surface $(\HH,h,0,\infty)$. It is therefore convenient to decide on a canonical way to choose $h$ from the set of possible fields, or a ``canonical parametrisation".  We will consider the \emph{last exit parametrisation} in most of this paper, since this parametrisation leads to the cleanest definition of
	$(2,2)$-quantum wedges. 
	Note that this is different from the \emph{unit circle parametrisation} considered in \cite{DMS18+}.
	\begin{defn}
		The last exit (resp., unit circle) parametrisation of a doubly-marked $\gamma$-quantum surface $\cS$ with the topology of $\HH$, is defined to be the representative $(\HH, h, 0,\infty)$ of $\cS$ such that if $h_{\rad}(r)$ is the average of $h$ on the semi-circle of radius $r$ around $0$ (i.e., $h_{\rad}$ is the projection of $h$ onto $H_1(\HH)$), then $s\mapsto h_{\rad}(e^{-s})-Q_\gamma s$ hits $0$ for the \emph{last} (resp., \emph{first}) time at $s=0$.
		\label{def:lastexit}
	\end{defn}
	If the last exit parametrisation of a surface exists 
	(i.e., if $h_\rad(r)+Q_\gamma \log r\neq 0$ for all $r>0$ small enough) it can easily be seen to be unique, by mapping the surface to the strip $\strip$ \en{with the map $\phi$ from \eqref{eqn:striptoH}}. Let $h_{\cir}=h-h_{\rad}$ be the projection of $h$ onto $H_2(\HH)$, and write $h_{\cir}^{\GFF}$ for the law of this field when $h$ is a Neumann GFF on $\HH$. Observe that this describes the law of a well-defined element of $\Hloc(\HH)$ (i.e., not only an element up to an additive constant).
	\begin{defn}
		Let $\gamma\in (0,2]$ and $\alpha<Q_\gamma$.  Then 
		the $(\gamma,\alpha)$-quantum wedge is the doubly-marked $\gamma$-quantum surface whose last exit parametrisation $(\HH,h,0,\infty)$ can be described as follows: 
		\begin{compactitem}
			\item $(h_{\rad}(e^{-s})))_{s\ge 0}$ has the law of $(B_{2s}+\alpha s)_{s\ge 0}$ conditioned to stay below $(Q_\gamma s)_{s\ge 0}$ for all time, where $B$ is a standard Brownian motion with $B_0=0$.
			\item $(h_{\rad}(e^{-s}))_{s\le 0}$ has the law of $(\wh{B}_{-2s}+\alpha s)_{s\le 0}$, where $\wh{B}$ is a standard Brownian motion with $\wh{B}_0=0$.
			\item $h_{\cir}$ is equal in law to $h^{\GFF}_{\cir}$.
			\item $h_{\cir}$, $(h_{\rad}(e^{-s}))_{s\ge 0}$, and $(h_{\rad}(e^{-s}))_{s<0}$ are independent. 
		\end{compactitem}
		\label{def:wedge}
	\end{defn}
	
	\begin{remark0}
		\label{rmk:dmswedge}
		In \cite{Sh16,DMS18+} the $(\gamma,\alpha)$-quantum wedge is defined to be the $\gamma$-quantum surface whose unit circle parametrisation is given by $(\HH,h,0,\infty)$, where:  $h=h_\cir+h_\rad$; $h_\cir$ is as in Definition \ref{def:wedge}; $h_\cir$ and $h_\rad$ are independent; and  $h_\rad(\e^{-s})$ is equal to $B_{2s}+\alpha s$ for $s\ge 0$, and to $\wh{B}_{-2s}+\alpha s$ conditioned to stay above $s\mapsto Q_\gamma s$ for $s<0$. 
		
		We show in Lemma \ref{prop:strip} below that this definition is equivalent to Definition \ref{def:wedge}
	\end{remark0} 
	
	In Definition \ref{def:wedge} we require $\alpha$ to be strictly smaller than $Q_\gamma$, and one can check that this is satisfied for $\alpha=\gamma$ when $\gamma\in(0,2)$. However, we are also interested in the case $\gamma=2$, where we have $Q_\gamma=2=\gamma$. Thus, we need to give a definition of the following surface, which arises as a limit of a $(\gamma,\gamma)$-quantum wedge when $\gamma\uparrow 2$.\footnote{More precisely, if $h$ is the field of a $(2,2)$-quantum wedge in the last exit parametrisation, and for $\gamma\in(0,2)$, $h_\gamma$ is the field of a $(\gamma,\gamma)$-quantum wedge in the last exit parametrisation, then $h_\gamma \to h$ in law  as $\gamma\uparrow 2$. To see this, it is easiest to map the surfaces to the strip $\strip$ with marked points at $\pm\infty$.} 
	\begin{defn}
		We define the $(2,2)$-quantum wedge to be the doubly-marked $2$-quantum surface whose last exit parametrisation $(\HH,h,0,\infty)$ can be described as follows: 
		\begin{compactitem}
			\item $(h_{\rad}(e^{-s}))_{s\ge 0}$ has the law of $(-\Bessel_{2s}+2s)_{s\ge 0}$, where $\Bessel$ is a 3-dimensional Bessel process started from $0$.
			\item $(h_{\rad}(e^{-s}))_{s\le 0}$ has the law of $(\wh{B}_{-2s}+2s)_{s\ge 0}$, where $\wh{B}$ is a Brownian motion started from $0$.
			\item $h_{\cir}$ is equal in law to $h^{\GFF}_{\cir}$.
			\item $h_{\cir}$, $(h_{\rad}(e^{-s}))_{s\ge 0}$, and $(h_{\rad}(e^{-s}))_{s<0})$ are independent. 
		\end{compactitem}
		\label{def:critical-wedge}
	\end{defn}
	
	The $(\gamma,\gamma)$-quantum wedges are of particular interest since they may be obtained by sampling a point from the boundary $\gamma$-LQG measure and then ``zooming in'' near this point. This was established in \cite{Sh16} for $\gamma\in(0,2)$, and Proposition \ref{prop:zoom} below is a variant of this result for $\gamma=2$.
	
	\begin{remark0}
		The last exit parametrization is more convenient than the unit circle parametrization for the $(2,2)$-quantum wedge since with the unit circle parametrization any neighborhood of zero has infinite mass a.s. This can be seen by using that with the unit circle parametrization, the field $(h_{\rad}(e^{-s}))_{s\ge 0}$ has the law of $(B_{2s}+2s)_{s\geq 0}$ for $B$ a standard Brownian motion started from 0.
	\end{remark0}
	
	In some of our proofs it will be convenient to parametrise the quantum wedges by the strip $\strip$ instead of the upper half-plane $\HH$. Recall that $H(\strip)$ denotes the Hilbert space closure of the subspace of functions $f \in C^\infty(\strip)$ with
	$\|f\|_\nabla:=(f,f)_\nabla<\infty$, 
	defined modulo additive constant. By
	\cite[Lemma 4.2]{DMS18+},
	$H(\strip)=H_1(\strip) \oplus H_2(\strip)$
	is an orthogonal 
	decomposition of $H(\strip)$, where $H_1(\strip)$ is the subspace of functions $f\in H(\strip)$ 
	that are constant on all line segments $\{x\}\times [0,\pi]$ for $x\in\RR$ (considered modulo an additive constant), and 
	$H_2(\strip)$ is the subspace of functions $f\in H(\strip)$ that have mean zero on all such line segments. Let $h^{\GFF,\strip}_{\cir}$ denote a field with the law of a Neumann GFF on $\strip$ projected onto $H_2(\strip)$ (as in the case of $\HH$, this is a well-defined element of $\Hloc(\strip)$). The strip is convenient to work with since the term $Q_\gamma \log|\phi'|$ in the coordinate change formula \eqref{eqn:coc} is equal to zero for conformal transformations of the kind $z\mapsto z+a$ for $a\in\RR$ (these are precisely the conformal maps from $\strip$ to itself that map $+\infty$ to $+\infty$ and $-\infty$ to $-\infty$, and correspond after conformal mapping to dilations of $\HH$). Furthermore, as the following remark illustrates for the case of the $(2,2)$-quantum wedge, the quantum wedges defined above have a somewhat nicer description when parametrised by the strip.
	\begin{remark0}\label{rmk::stripwedge}
		The surface $(\strip,h,\infty,-\infty)$ with $h\eqD h^{\GFF,\strip}_{\cir}+h_{\rad}$ has the law of a $(2,2)$-quantum wedge, if $h_{\rad}\in H_1(\strip)$ (viewed as a distribution modulo an additive constant) and the following hold:
		\begin{compactitem}
			\item $(h_{\rad}(s))_{s\ge 0}$ has the law of $-\Bessel_{2s}$ where $\Bessel$ is a 3-dimensional Bessel process starting from $0$. 
			\item $(h_{\rad}(s))_{s\le 0}$ has the law of ${B}_{-2s}$, where ${B}$ is a standard Brownian motion starting from $0$.
			\item $h^{\GFF,\strip}_{\cir}$, $(h_{\rad})_{s\ge 0}$, and $(h_{\rad}(s))_{s\le 0})$ are independent. 
		\end{compactitem}
		\label{rmk:wedge-strip}
	\end{remark0}
	
	\begin{lemma}
		For $\gamma\in(0,2)$ the definition of a $(\gamma,\alpha)$-quantum wedge in Definition \ref{def:wedge} is equivalent to the definition given in \cite[Definition 4.5]{DMS18+} (see Remark \ref{rmk:dmswedge}).
		\label{prop:strip}
	\end{lemma}
	\begin{proof}
		Let $(\HH,h_u,-\infty,\infty)$ be the \cite{DMS18+} definition of a quantum wedge, as in Remark \ref{rmk:dmswedge}. 
		Let $(\HH,h_\ell,-\infty,\infty)$ be a $(\gamma,\alpha)$-quantum wedge with the last exit parametrization as in Definition \ref{def:wedge}. Also let $\phi:\strip\to\HH$ be defined by \eqref{eqn:striptoH}, and observe that $\{0 \}\times[0,\pi]$ is mapped to the unit semi-circle under this map.
		
		Define 
		$\wh h_{\uo}=h_{\uo}\circ\phi+Q_\gamma\log|\phi'|$ and 
		$\wh h_{\ell}=h_\ell\circ\phi+Q_\gamma \log|\phi'|$, and let 
		$\wh h_{\uo}=\wh h_{\uo,\rad}+\wh h_{\uo,\cir}$ and $\wh h_{\ell}=\wh h_{\ell,\rad}+\wh h_{\ell,\cir}$ be the orthogonal decompositions of these fields. Then $\wh h_{\uo,\cir}$ and
		$h_{\ell,\cir}$ are both equal in distribution to $h_\cir^{\GFF,\strip}$. For $B$ a standard Brownian motion, $(\wh h_{\ell,\rad}(s))_{s\ge 0}$ has the law of $(B_{2s}+(\alpha-Q) s)_{s\ge 0}$ conditioned to be negative for $s>0$, and $(\wh h_{\ell,\rad}(s))_{s\le 0}$ has the law of $(B_{-2s}+(\alpha-Q) s)_{s\le 0}$. Furthermore, $(\wh h_{\uo,\rad}(s))_{s\ge 0}$ has the law of $(B_{2s}+(\alpha-Q) s)_{s\ge 0}$, and $(\wh h_{\uo,\rad}(s))_{s\le 0}$ has the law of $(B_{-2s}+(\alpha-Q) s)_{s\le 0}$ conditioned to be positive for $s<0$. Let $a=\inf\{t\leq 0\,:\, \wh h_{\ell,\rad}(t) <0 \}$. We conclude by observing that if we apply the change of coordinates $z\mapsto z-a$ to the field $\wh h_{\ell}$, 
		we get a field with the law of $\wh h_{\uo}$\en{; this can e.g.\ be deduced from the last assertion of \cite[Lemma 3.4]{RV-tail-GMC} and \cite[Remark 3.5]{RV-tail-GMC}, which refers to \cite{Wil74}}. %\ec{This just a simple property of Brownian motion with drift: see for example \cite{Wil74}} \nina{Did you have in mind a particular statement of Wil74 here?}.
	\end{proof}

	\subsection{Gaussian multiplicative chaos and the Liouville measures}\label{sec:gmc}

	In this section, we give a proper definition of the \emph{boundary} Liouville quantum gravity measures described in the introduction. For a much more complete survey, including the case of bulk LQG measures, we refer the reader to \cite{RV10,DS11,Ber17} for the subcritical case and to \cite{RV14,DRSV14one,DRSV14two,Pow18chaos} for the critical case. 
	
	In the following, when we refer to the topology of \emph{local weak convergence} for measures on $\RR$, we mean the topology such that $\mu_n\to \mu$ iff $\mu_n|_{[-R,R]}\to \mu|_{[-R,R]}$ weakly as measures on $[-R,R]$ for every $R>0$.
	
	The following statement comes from \cite{DS11} when $\gamma\in (0,2)$, and from 
	\cite{Pow18chaos} when $\gamma=2$ (with a trivial adaptation of the argument from the bulk to the boundary measure). 
	\begin{lemma}\label{lem::defnlm}
		Suppose that $\gamma\in (0,2]$ and let $h$ be a Neumann GFF in $\HH$ with some fixed choice of additive constant, or a GFF with mixed boundary conditions on $\D^+=\D\cap \HH$ (free on $\partial \D^+\cap \R$, and Dirichlet with some $\rho$ on  $\partial \D^+ \setminus \R$). Let $h_\eps$ denote the $\eps$ semi-circle average field of $h$ on $\RR$,\footnote{That is, $h_\eps(z)=(h,\rho_\eps^z)$ where $\rho_\eps^z$ is uniform measure on the semi-circle (contained in $\HH$) of radius $\eps$ around $z$.} let $dz$ denote Lebesgue measure on $\RR$, and set
		\begin{align}
		\nu_{h,\eps}^{\gamma}(dz) & =\exp\Big(\frac{\gamma}{2}h_\eps(z)\Big)\eps^{\gamma^2/4} \, dz & \gamma\in (0,2),\\
		\nu_{h,\eps}(dz) & = \Big(-\frac{h_\eps}{2} +\log(1/\eps)\Big)\exp(h_\eps(z))\eps \, dz & \gamma=2.
		\end{align}
		Then $\nu_{h,\eps}^\gamma$ converges in probability to a limiting measure $\nu_h^\gamma$ (resp., $\nu_{h,\eps}$ converges in probability to a limiting measure $\nu_h$ when $\gamma=2$) as $\eps\to 0$. These convergences are with respect to the topology of local weak convergence of measures on $\RR$.
	\end{lemma} 
	
	\begin{lemma}\label{lem::defnlmqw}
		The result of Lemma \ref{lem::defnlm} also holds when $h$ is the field of a $(\gamma',\alpha)$-quantum wedge in the last exit parametrisation, with $2\ge \gamma'\ge \gamma$ and $\alpha<Q_{\gamma'}\le Q_\gamma$.
	\end{lemma}
	
	Note that we do not require $\gamma'=\gamma$ here. We need to work in this set-up in, for example, Lemma \ref{lem::cd_critwedge}. 
	
	\vspace{0.1cm}
	\begin{proof} For notational simplicity we work in the case $\gamma\in(0,2)$, but the argument when $\gamma=2$ is the same. Without loss of generality, it suffices to show that $\nu_{h,\eps}^\gamma$ converges in probability, as a measure on $[-1,1]$, as $\eps \to 0$. To show this, we will explain how to obtain $h|_{\D}$ from a field $h'$ that is absolutely continuous with respect to a Neumann GFF (by re-centring around a $\nu^{\en{\alpha}}$-typical point). The result then follows from Lemma \ref{lem::defnlm}.
		
		More precisely, we consider the following construction. Let $\mathbb{P}$ denote the law of a Neumann GFF on $\HH$, with additive constant fixed so that its average on the unit semi-circle is equal to $0$, and write $(h',z)$ for a pair with joint law 
		\[ \frac{ \I_{\{z\in [-1,1]\}}\nu^{\alpha}_{h'}(dz)}{\mathbb{E}_\P[\nu^{\alpha}_{h'}([-1,1])]} \, \P(dh'). \]
		Let $\wt h$ be the field $h'$ after re-centring around the point $z$, i.e., $\wt h=h'(\cdot + z)$.
		
		Then it follows from \cite[\S 6.3]{DS11} that if $\wt h_\rad$ is the projection of $\wt h$ onto $H_1(\HH)$ (and $\wt h_\rad(s)$ denotes its common value on the semi-circle of radius $\e^{-s}$ around $0$) then $(\wt h_\rad(s)-\wt h_\rad(0))_{s\ge 0}$ has the law of $(B_{2s}+\alpha s)_{s\ge 0}$, where $B$ is a standard Brownian motion. Moreover, by \en{scale} invariance of $h_\cir^{\GFF}$, the projection $\wt h_{\cir}$ of $\wt h$ onto $H_2(\HH)$ is equal in law to $h_\cir^{\GFF}$. 
		
		Now, let $M=\sup_{s\ge 0} \wt h_\rad(s)-Q_\gamma s $, and let $T$ be the time at which this maximum is achieved (these are both finite a.s.\ since $\alpha<Q_\gamma$ by assumption). Then by scale invariance of $h_{\cir}^{\GFF}$, if $\psi_T:\HH\to \HH$ is the map $z\mapsto \e^{-T}z$, the field $\wh h := \wt h \circ \psi_T - M$ restricted to $\D^+$ has the same law as $h$ restricted to $\D^+$. 
		
		From here we can conclude, by observing that the law of $h'$ is absolutely continuous with respect to that of a Neumann GFF in $\HH$. Therefore, since all that is done to get from $h'$ to $\wh h$ is to re-centre around a random point, rescale by a random amount, and subtract a random constant, Lemma \ref{lem::defnlm} implies that $\nu^\gamma_{\wh h,\eps}$ converges in probability as a measure on $[-1,1]$. By the previous paragraph, the same thing then holds for $h$.
	\end{proof}
	
	\begin{remark0}
		\label{rmk:propslm}
		The measures $\nu_h^\gamma$, $\nu_h$ defined in Lemmas \ref{lem::defnlm} and \ref{lem::defnlmqw} are a.s.\ atomless and give strictly positive mass to every interval of strictly positive length a.s.\ (see, for example, \cite{DS11,DRSV14one}).
	\end{remark0}
	
	\begin{remark0}
		The measures $\nu_{h'}^\gamma$ (resp., $\nu_{h'}$) can be defined using the same regularisation procedure whenever $h'=h\circ \phi + Q_\gamma \log |\phi'|$ (resp., $h'=h\circ\phi + 2 \log |\phi'|$ ) for $h$ as in Lemma \ref{lem::defnlm} or \ref{lem::defnlmqw} and some conformal map $\phi$. Equivalently, $\nu_{h'}^\gamma$ (resp., $\nu_{h'}$) can be defined as the push-forward of $\nu_h^\gamma$ (resp., $\nu_h$) by $\phi^{-1}$.
	\end{remark0}
	
	The following lemma will be important when we construct the critical quantum zipper by taking a limit of subcritical quantum zippers.
	
	\begin{lemma}\label{lem::cd_critwedge}
		Let $\gamma_n\uparrow 2$ as $n \to \infty$, and $h$ be a $(2,1)$-quantum wedge in the last exit parametrisation. Then 
		\[ \frac{\nu_h^{\gamma_n}}{2-\gamma_n}\to 2 \nu_h\]
		in probability as $n\to \infty$, with respect to the topology of local weak convergence of measures on $\RR$. 
	\end{lemma}
	
	\begin{proof}
		This was shown in \cite[\S 4.1.1-2]{APS18two} when $h$ is either one of the fields in the statement of Lemma \ref{lem::defnlm}. It extends to the case when $h$ is a $(2,1)$-quantum wedge by the same proof as for Lemma \ref{lem::defnlm} (using that it holds for the Neumann boundary condition GFF and then re-centring the field around a $\nu_h^{1}$-typical point).
	\end{proof}
	
	\subsection{Schramm--Loewner evolutions}\label{sec:sle}
	We assume the reader is familiar with the basic theory of Schramm--Loewner evolutions (SLE): for an introduction, see e.g.\ \cite{Law05,Kem17}. In this section we simply fix some notation and discuss a few points that will be relevant later on. 
	
	In this article, we will consider \emph{chordal} $\SLE_\kappa$ with $\kappa\in (0,4]$. $\SLE_\kappa$ in $\HH$ from $0$ to $\infty$ is defined to be the Loewner evolution in $\HH$ with random driving function $(W_t)_{t\ge 0}=(\sqrt{\kappa}B_t)_{t\ge 0}$, where $B$ is a standard Brownian motion. When $\kappa\in (0,4]$ an SLE$_\kappa$ is a.s.\ a simple curve that does not touch the real line. We usually parametrise an $\SLE_\kappa$ curve $\eta$ by \emph{half-plane capacity}; that is, we choose the parametrisation of $\eta$ such that for every $t>0$, the unique conformal map $\wt g_t:\HH\setminus \eta([0,t])\to \HH$ with $\wt g_t(z)=z+a_t/z + O(|z|^{-2})$ as $|z|\to \infty$ for some $a_t>0$, satisfies $a_t=2t$. We use the notation $g_t$ for the \emph{centred Loewner map} $g_t=\wt g_t-W_t$, that sends $\eta(t)$ to $0$.
	
	A curve $\eta$ between boundary points $a$ and $b$ in a domain $D$ is said to be an $\SLE_\kappa$ from $a$ to $b$ if it is the image of an $\SLE_\kappa$ in $\HH$ from $0$ to $\infty$, under a conformal map from $ \HH$ to $ D$ mapping $0$ to $a$ and $\infty$ to $b$. 
	
	\begin{defn}[Reverse $\SLE_{\kappa}$] A \emph{reverse Loewner evolution} with continuous driving function $W_t: [0,\infty)\to \R$ is a solution $\wt{f}(t,z)=\wt{f}_t(z)$ to the following differential equation for every $z\in \HH$: 
		\eqbn
		\partial_t \wt f_t(z)=\frac{-2}{\wt f_t(z)-W_t}; \;\;\;\; \wt f_0(z)=z.
		\eqen
		In fact for every $z\in \HH$ (see e.g.\ \cite[Lemma 4.9]{Kem17}), a solution exists for all $t\ge 0$, so that each $\wt f_t$ defines a map $\HH\mapsto \wt f_t(\HH)$. 
		
		A reverse $\SLE_{\kappa}$ flow is the reverse Loewner evolution $(\wt f_t)_{t\ge 0}$ driven by $W_t=\sqrt{\kappa} dB_t$, where $B$ is a standard Brownian motion. One can also consider the \emph{centred} reverse $\SLE_{\kappa}$ flow, defined by $f_t(z)=\wt f_t(z+W_t)$ for all $z,t$. Then $(f_t)_{t\ge 0}$ satisfies the following SDE for all $z\in \HH$:
		\eqb
		df_t(z)=\frac{-2}{f_t(z)}dt-dW_t \;\;\;\; f_0(z)=z.
		\label{eqn:rsle}
		\eqe
		Moreover, there a.s.\ exists a continuous curve $\eta$ such that for each $t$ we have $\HH\setminus f_t(\HH)=\eta([0,t])$.
		\label{def:rsle}
	\end{defn}
	Due to the time-reversal property of Brownian motion, if $(f_t)_{t\ge 0}$ is a centred reverse $\SLE_{\kappa}$ and $(g_t)_{t\ge 0}$ is a centred forward $\SLE_{\kappa}$, both parametrised by half-plane capacity, then for any \emph{fixed} $t\geq 0$, $f_t^{-1}$ is equal in law to $g_t$. In other words, if $t>0$ is fixed, $\eta([0,t])$ is a forward $\SLE_{\kappa}$ run until it has half-plane capacity $t$ and $\eta'([0,t])$ is a reverse $\SLE_{\kappa}$ run until it has half-plane capacity $t$, then $\eta([0,t])$ is equal in law to  $\eta'([0,t])$.
	\vspace{0.1cm}
	
	Let us now provide a notion of convergence for Loewner evolutions; this will be particularly important in our construction of the critical conformal welding. Note that when considering sequences $(f_n)_{n\in \N}$ or $(g_n)_{n\in \N}$ of Loewner evolutions, we move the time parameter $t$ into a superscript.
	
	\begin{defn}\label{def::cart_conv}
		Suppose that $(f^t_n)_{t\ge 0}$ for $n\in \N$ and $(f^t)_{t\ge 0}$  are centred, reverse Loewner evolutions in $\HH$ from $0$ to $\infty$, parametrised by half-plane capacity. Let $\sigma_n: \RR\to [0,\infty)$ be defined by setting 
		$ \sigma_n(x)=\inf \{t\ge 0: f^t_n(x)=0\}$ for each $x\in \RR, n\in \N $,
		and define $\sigma$ in the corresponding way for $f$. Then we say that $f_n\to f$ \emph{in the Carath\'{e}odory+ topology} if 
		\begin{compactitem}
			\item for every $T<\infty$ and $\eps>0$, $f_n$ converges to $f$ uniformly on $[0,T]\times \{\HH+\im \eps\}$; and 
			\item $\sigma_n\to \sigma$ uniformly on compacts of $\RR$. 
		\end{compactitem} 
	\end{defn}
	
	\begin{remark0}
		Note that this is stronger than the usual notion of Carath\'{e}odory convergence for Loewner evolutions. For \emph{forward} Loewner evolutions, Carath\'{e}odory convergence is characterised by the requirement that, if $g_n,g$ are the flows in question, we have $g_n^{-1}\to g^{-1}$ uniformly on $[0,T]\times \{\HH+\im \eps\}$ for every $T,\eps>0$ (see \cite[\S 4.7]{Law05}). The motivation for working with this stronger topology should be clear from the nature of the conformal welding problem that we are considering.
	\end{remark0}
	
	In the sequel we make the following slight abuse of notation. Suppose we have $(\eta_n)_{n\in \N}$ and $\eta$, a collection of simple, continuous, transient curves starting from $0$ in $\HH$. Then we will say that $\eta_n\to \eta$ in the \emph{Carath\'{e}odory} topology, if the corresponding forward (half-plane capacity parametrised) Loewner evolutions converge in the Carath\'{e}odory sense.
	
	The convergence results that will be important in this article are the following.
	
	\begin{lemma}
		\label{lem:eta_conv}
		Suppose that $\kappa_n\uparrow 4$ as $n\to \infty$, that $\eta_n$ has the law of an $\SLE_{\kappa_n}$ curve in $\HH$ from $0$ to $\infty$ for each $n\in \N$, and that $\eta$ has the law of an $\SLE_4$ in $\HH$ from $0$ to $\infty$. Then $\eta_n\to \eta$ in distribution as $n\to \infty$, with respect to the Carath\'{e}odory topology.
	\end{lemma}
	
	\begin{proof} See \cite[Lemma 6.2]{Kem17} \end{proof}

	\begin{lemma}\label{lem::df_conv}
		Suppose that $\kappa_n\uparrow 4$ as $n\to \infty$, and that $f_n$ is a centred, reverse $\SLE_{\kappa_n}$ in $\HH$ from $0$ to $\infty$ for each $n$. Let $f$ be a centred, reverse $\SLE_4$ in $\HH$ from $0$ to $\infty$. Then $f_n$ converges to $f$ in distribution, with respect to the Carath\'{e}odory+ topology. 
	\end{lemma}
	
	\begin{proof} 
		For the proof we couple together $((f_n)_{n\in \N},f)$, by setting their driving functions equal to $((\sqrt{\kappa_n}B)_{n\in \N},2B)$, where $B$ is a single standard Brownian motion. Then by \cite[Proposition 6.1]{Kem17} we have that $f_n\to f$ uniformly a.s.\ on $[0,T]\times \{\HH+\im\eps\}$, for any $T,\eps>0$. 
		
		To show the convergence of $\sigma_n$ (as in Definition \ref{def::cart_conv}), we define $(h^t_n(x))_{t\ge 0}:=((\frac{\sqrt{\kappa_n}}{2})^{-1}f_n^t(\frac{\sqrt{\kappa_n}}{2}x))_{t\ge 0}$ for each $n$ and $x\in \RR$ so that
		\begin{align}\label{eqn::monotone_bessel}
		dh^t_n(x) & = \frac{-2(\frac{4}{\kappa_n})}{h^t_n(x)}\, dt - 2\,dW_t \;\text{ for all } t\le \sigma_n^*(x); & h^0_n(x)=x,\nonumber \\ 
		df_t(x) & =  \frac{-2}{f_t(x)}\, dt - 2\,dW_t  \;\;\;\;\;\text{ for all } t\le \sigma(x); & f_0(x)=x,
		\end{align}
		where $\sigma^*_n(x):=\sigma_n(\frac{\sqrt{\kappa_n}}{2}x)$. 
		
		We will first show that $\sigma_n^*\to \sigma$ uniformly a.s.\ on compacts of time. Observe that the coupled equations \eqref{eqn::monotone_bessel} imply that for any fixed $x\in \RR$, $\sigma_n^*(x)$ is a.s.\ increasing in $n$ and bounded above by $\sigma(x)$, so has some a.s.\ limit $\sigma^*(x)\le \sigma(x)$. In fact, it holds that $\sigma^*(x)=\sigma(x)$ a.s. To see this, without loss of generality assume that $x\ge 0$ and suppose for contradiction that $\sigma(x)>\sigma^*(x)$. This means that for some $\eps>0$ we have $f^t(x)\ge \eps$ for all $t\le \sigma^*(x)$. Define $\sigma^\eps_n(x)$ to be the first time that $h_n^t(x)\le \eps/2$ for each $n$, so that: \begin{compactitem}
			\item $\sigma^\eps_n(x)\le \sigma_n^*(x)\le \sigma^*(x)$ for all $n$; and 
			\item $h^t_n(x),f^t(x)\ge \eps/2$ for all $t\le \sigma_n^\eps(x)$ and all $n$.
		\end{compactitem} Then \eqref{eqn::monotone_bessel}, plus Gr\"onwall's inequality applied to the function $h_n^t-f^t$, implies that 
		$|h_n^{\sigma_n^\eps(x)}(x)-f^{\sigma_n^\eps(x)}(x)|\to 0$ 
		as $n\to \infty$. This is a contradiction, since the first term in the difference is equal to $\eps/2$ by definition, and the second should always be greater than $\eps$.
		
		For any $K>0$, this argument then gives the existence of a probability one event $\Omega_0$, on which we have $\sigma_n^*(q)\to \sigma(q)$ for all $q\in \Q\,\cap\,[0,K+1]$. Since $\sigma$ and $(\sigma_n)_{n\in \N}$ are defined from reverse $\SLE_\kappa$ curves, we may also assume that $\sigma$ and $(\sigma_n)_{n\in \N}$ are continuous on $\Omega_0$. 
		So now, suppose we are working on $\Omega_0$, and take any $x\in [0,K]$. Let $q_k^-\uparrow x$ and $q_k^+\downarrow x$ with $q_k^{\pm} \in \Q\cap[0,K+1]$ for every $k$, so that $\sigma_n^*(q_k^-)\le \sigma_n^*(x)\le \sigma_n^*(q_k^+)$ for every $n$, and $\sigma^*_n(q_k^-)\uparrow \sigma(q_k^-)$, $\sigma^*_n(q_k^+)\uparrow \sigma(q_k^+)$ as $n\to \infty$ for every $k$. This means that $\sigma^*_n(x)$ is a bounded sequence, and any converging subsequence has limit lying between $\sigma(q_k^-)$ and $\sigma(q_k^+)$ for every $k$. Since $\sigma$ is continuous, this implies that any such subsequential limit must be equal to $\sigma(x)$, and so in fact, it must be that $\sigma_n^*(x)\to \sigma(x)$. To summarise, on this event $\Omega_0$ of probability one, we have that: $\sigma^*_n\to \sigma$ pointwise on $[0,K]$; $\sigma^*_n(x)$ is increasing in $n$ for every $x\in [0,K]$; and the functions $\sigma$ and $(\sigma^*_n)_{n\in \N}$ are continuous. These are exactly the conditions of Dini's theorem, and so we may deduce that $\sigma_n^*\to \sigma$ uniformly on $[0,K]$ a.s.\
		
		To finish the proof, it is enough to show that for $K'$ arbitrary, the quantity $\sup_{x\in [0,K']}|\sigma_n(x)-\sigma(x)|$ converges to $0$ a.s.\ as $n\to \infty$. Suppose without loss of generality that $\kappa_n\ge 2$ for all $n$. Then setting $K=2K'$ in the previous paragraph, one deduces the existence of a probability one event $\Omega_1$, on which $\sup_{y\in [0,2K']}|\sigma_n^*(y)-\sigma(y)|\to 0$ as $n\to \infty$ and $\sigma$ is continuous. Then we have
		\eqbn\begin{split}
			\sup_{x\in [0,K']}|\sigma_n(x)-\sigma(x)|  
			&= \sup_{x\in[0,K']} |\sigma_n^*((2/\sqrt{\kappa_n})\, x)-\sigma(x)|  \\
			&\leq \sup_{x\in[0,K']} |\sigma_n^*((2/\sqrt{\kappa_n})\, x)-\sigma((2/\sqrt{\kappa_n})\, x)| 
			+ |\sigma((2/\sqrt{\kappa_n})\, x)-\sigma(x)|\\
			&\leq \sup_{y\in[0,2K']} |\sigma_n^*(y)-\sigma(y)| 
			+ \sup_{x\in[0,K']}|\sigma((2/\sqrt{\kappa_n})\, x)-\sigma(x)|,
		\end{split}
		\eqen	
		and on $\Omega_1$, the final expression goes to $0$. This completes the proof.
	\end{proof}
	
	\section{The $(2,2)$-wedge via ``zooming in'' at quantum-typical point}
	\label{sec:zoom}
	
	The main goal of this section is to prove Proposition \ref{prop:zoom} below. This proposition illustrates why the $(2,2)$-quantum wedge is a particularly natural quantum surface,  and will also be important in our proof of Theorem \ref{thm1}. Before we state this proposition, we briefly define the relevant notion of convergence for $\gamma$-LQG surfaces. Let $\cS_n$ for $n\in\N$ and $\cS$ be doubly-marked $\gamma$-quantum surfaces with the topology of $\HH$. We say that $\cS_n$ converges to $\cS$ in the sense of doubly-marked $\gamma$-quantum surfaces if we can find parametrisations $(D,h_n,a,b)$ and $(D,h,a,b)$ of $\cS_n$ and $\cS$, respectively, with $D\subsetneq \C$ and $a,b\in\ol D$, such that for any open and bounded $U\subset D$, $h_n|_U$ converges to $h|_U$ in  $H^{-1}(U)$.\footnote{We remark that convergence of quantum surfaces is defined somewhat differently in \cite{Sh16} and \cite{DMS18+} than in the current paper. In \cite{Sh16} one embeds the surfaces such that the field $h_n$ gives unit mass to the unit half-disk for all $n$, and the surfaces are said to converge if, restricted to any bounded subset of $\HH$, the area measures $\mu^\gamma_{h_n}$ associated with the $h_n$ converge weakly to the area measure $\mu^\gamma_h$ associated with $h$. 
		In \cite{DMS18+} one embeds the surfaces with the unit circle embedding and requires that the fields $h_n$ converge as distributions to $h$. 
		However, the exact notion of convergence considered does not play an important role in this paper, and the convergence results we prove also hold for the alternative notions of convergence considered in \cite{Sh16} and \cite{DMS18+}.}
	\begin{propn} 
		Let $D\subset\HH$ be a simply connected domain such that $\partial D\cap \RR$ contains an interval of positive length. Furthermore, assume there exists a conformal map $\phi:D\to\HH$ such that the derivative $\phi'$ extends continuously to $\partial D\cap\RR$ and is non-zero on $D\cup
		\{\partial D\cap\RR\}$. Let $h$ be an instance of the GFF with continuous Dirichlet boundary conditions on $\partial D\setminus\RR$ and free boundary conditions on $\partial D\cap \RR$. Let $I=(a,b)\subset\partial D\cap\RR$ be a bounded interval, and let $z_0$ be an arbitrary fixed point of $\partial D\setminus I$. Finally, sample $z$ uniformly from $\nu_h$ restricted to $I$ (renormalised to be a probability measure).
		
		Then as $C\rta\infty$, conditioned on the location of $z$ and on $\nu_h([a,z]),\nu_h([z,b])$, the random quantum surface $(D,h+C,z,z_0)$ converges  in law \en{with probability one with respect to $z,\nu_h([a,z]),\nu_h([z,b])$?} (in the sense of doubly-marked $2$-quantum surfaces) to a $(2,2)$-quantum wedge. 
		\label{prop:zoom}
	\end{propn}
	
	\begin{remark0}
		\en{Note that the above is a statement about the law of a quantum surface conditionally on several quantities. The same statement holds unconditionally, but we need the stronger statement for the proof of Theorem \ref{thm1}. Let us now briefly explain why.} 
		
		\en{Theorem \ref{thm1} says that when we cut a $(2,1)$-quantum wedge with an independent SLE$_4$, the surfaces on either side are independent $(2,2)$-quantum wedges. For the proof the idea is to make use of the stationary critical zipper, Theorem \ref{thm:criticalzipper}. This can be used (by ``zipping down" the SLE$_4$ some amount of quantum length) to say that the law of two surfaces in question are the same as the law as the surfaces $\lim_{C\to \infty}(\HH,h+C,\cX,\infty)$ and $\lim_{C\to \infty}(\HH,h+C,\cY,\infty)$ where $\cX,\cY$ are two quantum-typical points at equal quantum distance to the right and left of 0. See Proposition \ref{prop:zoom-twopoints} and the proof of Theorem \ref{thm1} below. In particular $\cX$ and $\cY$ depend on one another via the quantum boundary length measure. It is therefore important to know that the local convergence to a quantum wedge described in Proposition \ref{prop:zoom} holds even given this information.}
	\end{remark0}
	
	We first prove a lemma that says, roughly speaking, that convergence of the type considered in Proposition \ref{prop:zoom} only depends on the local behaviour of the field $h$ around the point $z$. This will be useful several places in what follows.
	\begin{lemma}
		Consider the setting of Proposition \ref{prop:zoom}, but now with arbitrary $h\in \Hloc(\wh D)$, where
		$\wh D\subset\HH$ is a simply connected domain containing $D$. Assume further that the boundary measure $\nu_h$ is well-defined on $I$ (as in Lemma \ref{lem::defnlm}), and a.s.\ assigns positive and finite mass to every subinterval of $I$ with strictly positive length. Finally, let $\wh z_0$ be an arbitrary fixed point on $\partial\wh D\setminus I$. Then the following statements are equivalent: 
		\vspace{0.1cm}
		\begin{compactitem}
			\item[(i)] conditioned on the location of $z$, and on $\nu_h([a,z]),\nu_h([z,b])$, the random quantum surface \\ $(D,h|_D+C,z,z_0)$ converges in law to a $(2,2)$-quantum wedge as $C\rta\infty$; 
			\item[(ii)]  conditioned on the location of $z$, and on $\nu_h([a,z]),\nu_h([z,b])$, the random quantum surface \\ $(\wh D,h+C,z,\wh{z}_0)$ converges in law to a $(2,2)$-quantum wedge as $C\rta\infty$.
		\end{compactitem}
		\label{prop:wedge-transform}
	\end{lemma}
	\begin{proof}
		We may assume without loss of generality that $h\in H^{-1}(\wh D)$ (rather than $h\in\Hloc(\wh D)$) since the field of a $(2,2)$-quantum wedge restricted to any bounded set is in $H^{-1}(\wh D)$, so the considered fields must be in $H^{-1}(U)$ for some neighbourhood $U$ around $z$ in order for the assumed convergence to hold. We may also assume without loss of generality that $\wh D=\HH$ and $\wh{z}_0=\infty$. Consider a conformal map $\phi:\HH\to D$ sending $0 \mapsto z$ and $\infty \mapsto z_0$. Without loss of generality, upon replacing $\phi$ by $\phi(c\, \cdot)$ for an appropriate $c>0$, we may assume that $\phi'(0)=1$. We only prove that (ii) implies (i), since the other direction can be verified by a similar argument. 
		
		Suppose that (ii) holds, and write $\wt h$ for a random element of $H^{-1}(\HH)$, with the law of $h(\cdot + z)$ conditionally on $(z,\nu_h([a,z]),\nu_h([z,b]))$. Then for every $C>1$ there exists a random conformal map $\psi_C:\HH\to \HH$ of the form $w\mapsto r_C\,w$ for $r_C>0$, such that $\wt h\circ\psi_C+2\log|\psi'_C|+C$ converges in law in $H^{-1}(\HH)$ as $C\to \infty$, to the field described in Definition \ref{def:critical-wedge}. Note that $r_C\to 0$ as $C\to \infty$ since when $C\rta\infty$ the measure assigned to any fixed boundary segment by $\wt h\circ\psi_C+2\log|\psi'_C|+C$ goes to infinity, while the measure assigned to (say) $[-1,1]$ by the field in Definition \ref{def:critical-wedge} is of order 1. 
		
		By the definition of convergence for doubly-marked 2-LQG surfaces, in order to prove (i) it is sufficient to show convergence of the following quantum surface to a $(2,2)$-quantum wedge:
		$$
		(\HH,\wt h \circ\phi +2\log|\phi'|+C,0,\infty),
		$$
		where we note that the field depends only on the restriction of $h$ to $D$).
		Equivalently, letting $h_\w$ denote the field in Definition \ref{def:critical-wedge}, it is sufficient to show the existence of maps $\wt\psi_C:\HH\to\HH$ of the same form as $\psi_C$ such that the convergence in law
		\eqb
		\wt h\circ\phi\circ\wt\psi_C +2\log|\phi'\circ\wt\psi_C|+2\log|\wt\psi'_C|+C\Rightarrow h_\w.
		\label{eq:transformed-field}
		\eqe
		holds in $H^{-1}(\HH)$ as $C\to \infty$.
		We will show that this in fact holds with $\wt\psi_C=\psi_C$. 
		
		To do this, we set $\wt h_C := \wt h\circ \psi_C+2\log|\psi_C'|+C$ and rewrite the left-hand side of \eqref{eq:transformed-field} as 
		\[ \wt h_C \circ \psi_C^{-1}\circ \phi \circ \psi_C + 2\log|\phi'\circ \psi_C|, \]
		where we can immediately note (since $\phi'(0)=1$, $\phi'$ is continuous, and $r_C\rta 0$) that the second term converges to 0 in distribution as $C\to \infty$. Furthermore, $\wt h_C$ is equal in distribution to $h_\w+g_C$ where  $g_C\Rightarrow 0$ in $H^{-1}(\HH)$ as $C\rta\infty$. 
		Defining $\wt\phi_C=\psi_C^{-1}\circ\phi\circ\psi_C$, in order to conclude the proof it is therefore sufficient to show that
		\eqbn
		(i)\,\,h_\w\circ\wt\phi_C\Rightarrow h_\w,\qquad 
		(ii)\,\,g_C\circ\wt\phi_C\Rightarrow 0,
		\eqen
		as $C\rta\infty$, whenever $(h_\w, \wt \phi_C)$ and $(g_C,\wt \phi_C)$ are coupled such that the marginal laws of $h_\w, g_C$ and $\wt \phi_C$ are as in the discussion above. Observe that $\wt\phi_C-z$ and its first derivatives converge to $0$ in probability, uniformly on compact subsets of $\HH\cup\RR$ as $C\rta\infty$. 
		
		Let $\mathcal{F}:= \{f\in H_0(\HH) \, : \, \|f\|_\nabla=1\}$ and recall that for an arbitrary $g\in H^{-1}(\HH)$, its $H^{-1}(\HH)$ norm is defined by
		\eqbn
		\|g\|_{H^{-1}(\HH)} = \sup\{ \en{( g,f )}\,:\,f\in \cF \}.
		\eqen
		To prove (ii), first note that
		\eqbn
		\|f\circ\wt\phi_C^{-1}\|_\nabla
		=\int [\partial_x( f\circ\wt\phi_C^{-1} )]^2 
		+ [\partial_y( f\circ\wt\phi_C^{-1} )]^2
		= \int [\partial_x f]^2(1+\xi_1) 
		+ [\partial_y f]^2(1+\xi_2),
		\eqen
		for some functions $\xi_1,\xi_2$ converging to $0$ in probability, uniformly on compact sets as $C\rta\infty$. Therefore the inequality $\|f\circ\wt\phi_C^{-1}\|_\nabla\leq 2\|f\|_\nabla$ holds with probability converging to $1$ as $C\to \infty$, uniformly on $\cF$.
		We now get (ii), since $\|g_C\|_{H^{-1}(\HH)}\Rightarrow 0$ as $C\rta \infty$, and
		\eqbn
		\sup_{\cF} \, \langle g_C\circ\wt\phi_C,f\rangle_\nabla
		= \sup_{\cF}  \, \langle g_C,f\circ \wt\phi_C^{-1} \rangle_\nabla
		\leq \sup_{\cF} \, \|g_C\|_{H^{-1}(\HH)}\cdot \|f\circ \wt\phi_C^{-1}\|_{\nabla}\le 2 \, \|g_C\|_{H^{-1}(\HH)}
		\eqen
		with probability converging to $1$ as $C\rta \infty$. 
		
		We also have that for some functions $\xi_1,\xi_2$ converging uniformly to zero in probability as $C\rta\infty$,
		\eqbn
		\|f\circ \wt\phi_C^{-1}-f\|_{\nabla}
		=\int [\partial_x f]^2 \xi_1
		+ [\partial_y f]^2\xi_2, 
		\eqen
		and this therefore converges to $0$ in probability as $C\rta \infty$, uniformly in $f\in \cF$. From this (i) follows since, uniformly in $f\in \cF$ and as $C\rta\infty$,
		\eqbn
		\langle h_\w\circ\wt\phi_C-h_\w,f \rangle_\nabla
		= \langle h_\w ,f\circ\wt\phi_C^{-1}-f  \rangle_\nabla
		\leq \|h_\w\|_{H^{-1}(\HH)} \cdot \|f\circ \wt\phi_C^{-1}-f\|_{\nabla}\, \Rightarrow\,  0.
		\eqen
	\end{proof}

	For $z\in I$ and $\ep>0$, define the semi-disk $\wh B(z,\ep)$ and $\ep_z\in(0,1]$ by 
	\eqbn
	\wh B(z,\ep):=B(z,\ep)\cap\HH,\qquad
	\ep_z=\sup\{ \ep\in(0,1]\,:\,\wh B(z,\ep)\subset D \}.
	\eqen
	Unless otherwise stated we assume throughout the section that $I$ is bounded away from $\HH\setminus D$ and, to simplify notation slightly, that
	\eqb
	\inf\{ \ep_z\,:\,z\in I \}>1.
	\label{eq:Idist}
	\eqe
	
	Let $h$ be a random generalised function with the law described in Proposition \ref{prop:zoom}; in the sequel, we denote the law of $h$ by $\P$. For $\ep\in(0,\ep_z)$ let $h_\ep(z)$ denote the average of $h$ on the semi-circle $\partial\wh B(z,\ep)\cap\HH$, and 
	for $\beta>1$ and $\ep\in(0,1]$, define the measure $d_{h,\eps}^\beta$ on $I$ by  
	\begin{equation}
	d_{h,\eps}^\beta(dz)=\Big(-\frac{h_\eps(z)}{2}+\log(1/\eps)+\beta\Big)\e^{h_\eps(z)}\eps\I_{\big\{\frac{h_\delta(z)}{2}<\log(1/\delta)+\beta \, \forall \delta\in [\eps,1]\big\}} \1_{z\in I} \, dz. 
	\label{eq:d-beta-eps}
	\end{equation}
	These measures played an important role in \cite{DRSV14one, DRSV14two, Pow18chaos}, and they are closely related to the \emph{derivative martingale} for the branching random walk (\cite{BK04}). The key point is that
	$d_{h,\eps}^\beta$ is a good approximation to the measure $\nu_{h,\eps}$ from Lemma \ref{lem::defnlm} when $\beta$ is large. It is however more convenient to work with, since its total mass is uniformly integrable in $\eps$ (which is not the case for $\nu_{h,\eps}$). More precisely, we have the following.
	\begin{lemma}\label{lem::d-ui}
		For any $A\subset I$ the family $(d_{h,\eps}^\beta(A))_{\eps\in(0,1]}$ is uniformly integrable (under $\P$).
	\end{lemma}
	\begin{lemma}\label{lem::d-nu}
		Denote by $\cC_\beta$ the event $\{\sup_{z\in I}\frac{h_\delta(z)}{2}<\log(1/\delta)+\beta \,\, ; \,\, \forall \delta\in [0,1]\}$. Then $\mathbb{P}(\cC_\beta)\to 1$ as $\beta\to \infty$.
	\end{lemma}
	The version of Lemma \ref{lem::d-ui} when the measures $d^\beta_{h,\ep}$ are defined in the bulk comes from \cite{Pow18chaos}, and the proof goes through in exactly the same way for the boundary measures \eqref{eq:d-beta-eps}. Lemma \ref{lem::d-nu} is a consequence of \cite{Aco14}. 
	
	\begin{remark0}\label{rmk::cb}
		On the event $\cC_\beta$ it holds that $d_{h,\eps}^\beta(dz)=\nu_{h,\eps}(dz)+\beta \eps \e^{h_\eps(z)}dz$. Moreover, (see \cite{RV10}) the measure $ \eps \e^{h_\eps(z)} dz$ converges to $0$ a.s.\ as $\eps\to 0$.
	\end{remark0}
	By uniform integrability of $d_{h,\eps}^\beta(I)$, we have the following.
	\begin{lemma}
		Let $\beta$ be fixed. Then the sequence $(h,d_{h,\eps}^\beta)$ is tight in $\eps$, with respect to the product topology formed from the topology of $H^{-1}_{\op{loc}}(D)$ in the first coordinate and the weak topology for measures on $D$ in the second coordinate.
		\label{lem:rooted1}
	\end{lemma}
	
	Let us take a subsequence of $\eps$ along which \[(h,d_{h,\eps}^\beta)\Rightarrow (h,d^\beta),\] and denote by $\P_\infty$ the law of the limiting pair. Note that the $\P_\infty$ marginal law of $h$ must be equal to its $\P$ law (as in Proposition \ref{prop:zoom}). Also write $d_h^\beta$ for the $\P_\infty$ conditional law of $d^\beta$ given $h$, which is a measurable function of $h$ by definition (although we will not need it, the proof of Lemma \ref{lem:rooted2} below actually shows that this function does not depend on the chosen subsequence). In fact, it should be the case that under $\P_\infty$, $d^\beta$ is  measurable with respect to $h$ (and so $d^\beta$ and $d_h^\beta$ are equal a.s.). However, for us it suffices to simply work with $d_h^\beta$. 
	
	\begin{remark0} \label{rmk::dbetanu} Observe that by Remark \ref{rmk::cb}, on the event $\mathcal{C}_\beta$ the convergence $d_{h,\eps}^\beta\to \nu_h$ holds in probability as $\eps\to 0$ (i.e., along any subsequence). Thus $d^\beta=d^\beta_h=\nu_h$ on this event.
	\end{remark0}
	
	The following elementary lemma will be used in the proof of Proposition \ref{prop:zoom}. It is straightforward to verify using Girsanov's theorem, the Markov property of Brownian motion, the reflection principle, and the fact that a 3-dimensional Bessel process started from a positive value is equal in law to a 1-dimensional Brownian motion started from that value and conditioned to stay positive. See, for example, \cite[Example 3]{Pin85}.
	\begin{lemma}
		Let $(B_t)_{t\geq 0}$ be a Brownian motion started from a possibly random position $B_0$ and let $\alpha:=\op{Var}(B_1-B_0)$ (so the Brownian motion has speed $\sqrt{\alpha}$). Let $\beta,\gamma>0$. Assume  $\P[B_0<\beta]>0$ and $\BB E[ |B_0|e^{\gamma B_0}]<\infty$. Then the following process $(M_t)_{t\geq 0}$ is a martingale (with respect to the natural filtration of $B$):
		$$
		M_t:=(-B_t+\gamma\alpha t+\beta)\1_{\{ -B_u+\gamma\alpha u + \beta>0\,\,\forall u\in[0,t] \}} e^{\gamma B_t-\frac{\gamma^2}{2}\alpha t}.
		$$
		For $t\geq 0$ let $\P_t$ denote the probability measure for which the Radon-Nikodym derivative relative to $\P$ is proportional to $M_t$. Define $X_u:=-B_u+\gamma\alpha u+\beta$ for $u\geq 0$. Under $\P_t$, the process $(X_u)_{u\leq t}$ has the following law.
		\begin{compactitem}
			\item $X_0$ has the law of $-B_0+\beta$ reweighted by $M_0=(-B_0+\beta)\1_{ \{-B_0+\beta>0\}} e^{\gamma B_0}$.
			\item Conditioned on $X_0$, $(X_u)_{u\in[0,t]}$ has the law of $(\cB_{\alpha s})_{u\in[0,t]}$, where $\cB$ is a 3-dimensional Bessel process started from $X_0$.
			\item Conditioned on $(X_u)_{u\in[0,t]}$, the process $(X_{u+t}-X_t)_{u\geq 0}$ has the law of $(B_u-B_0)_{u\geq 0}$.
		\end{compactitem}
		\label{prop:rn}
	\end{lemma}

	\begin{lemma}
		Let $h$ and $\beta$ be as in Lemma \ref{lem:rooted1}. Let $\Q$ denote the law of $h$ reweighted by $d^\beta_h(I)$, and define $g(z):=\Q[(-h_1(z)+\beta)\e^{h_1(z)}\I_{\{h_1(z)<\beta\}}]$. Note that under $\Q$, $d_h^\beta(I)$ is a.s.\ strictly positive. Then (i) and (ii) below give two equivalent procedures to sample a pair 
		\begin{equation}\label{eqn:hz_space} (\wh h,z) \text{ with } z\in I \text{ and } \wh h\in H^{-1}( \wh B(z,1) ).\end{equation} 
		\begin{compactitem}
			\item[(i)] Sample $h$ according to $\Q$, then sample $z$ from $d^\beta_h$ (normalised to be a probability measure), and set $\wh h=h|_{\wh B(z,1)}$.
			\item[(ii)] Sample $z$ from $I$ with density proportional to $g$ relative to Lebesgue measure, and then set $\wh h=\wh h_\cir+\wh h_\rad$, where $\wh h_\cir$ and $\wh h_\rad$ are independent, $\wh h_\cir$ has the law of the projection of $h$ onto $H_2(\wh B(z,1))$, 
			and $\wh h_\rad(x)=A_{-\log|x-z|}$ for a process $(A_s)_{s\geq 0}$ such that:
			\begin{compactitem}
				\item $A_0$ has the law of $h_1(z)$, reweighted by $(-h_1(z)+\beta)e^{h_1(z)}\I_{ \{ h_1(z)\leq\beta \} }$;
				\item conditioned on $A_0$, $(A_s)_{s\geq 0}$ is equal in distribution to $ (-\cB_{2s}+2s+\beta)_{s\geq 0}$ for $(\cB_s)_{s\geq 0}$ a 3-dimensional Bessel process started from $-A_0+\beta$.
			\end{compactitem}
		\end{compactitem}
		\label{lem:rooted2}
	\end{lemma}
	
	The proof of Lemma \ref{lem:rooted2} goes via an argument in the style of \cite{Sha16}.
	\medskip
	
	\begin{proof}
		Let $\Q_\ep$ be the law of $h$ reweighted by $d^\beta_{h,\ep}(I)$. 
		By Lemma \ref{prop:rn} and the definition of $d^\beta_{h,\ep}$, (i') and (ii') below give two equivalent procedures to sample a pair $(\wh h,z)$ as in \eqref{eqn:hz_space}. 
		\begin{compactitem}
			\item[(i')] Sample $h$ according to $\Q_\ep$, then sample $z$ from $d^\beta_{h,\ep}$ (normalised to be a probability measure), and set $\wh h=h|_{\wh B(z,1)}$.
			\item[(ii')] Sample $z$ from $I$ with density proportional to $g$ relative to Lebesgue measure, and then set $\wh h=\wh h_\cir+\wh h_\rad$, where $\wh h_\cir$ and $\wh h_\rad$ are independent, $\wh h_\cir$ has the law of the projection of $h$ onto $H_2(\wh B(z,1))$, 
			and $\wh h_\rad(x)=A_{-\log|x-z|}$ for a process $(A_s)_{s\geq 0}$ such that:
			\begin{compactitem}
				\item $A_0$ has the law of $h_1(z)$, reweighted by $(-h_1(z)+\beta)e^{h_1(z)}\I_{ \{ h_1(z)\leq\beta \} }$;
				\item conditioned on $A_0$,  $(A_s)_{s\in[0,\log\ep^{-1}]}$ is equal in distribution to $ (-\cB_{2s}+2s+\beta)_{s\in[0,\log\ep^{-1}]}$ for $(\cB_{s})_{s\geq 0}$ a 3-dimensional Bessel process started from $-A_0+\beta$;
				\item conditioned on $(A_s)_{s\in[0,\log\ep^{-1}]}$,  $(A_s)_{s\in[\log\ep^{-1},\infty)}$ is equal in distribution to $(B_{2(s-\log\ep^{-1})}+A_{\log\ep^{-1}})_{s\in[\log\ep^{-1},\infty)}$ for $(B_{s})_{s\geq 0}$ a standard Brownian motion started from 0.
			\end{compactitem}
		\end{compactitem} \vspace{0.1cm}
		It is clear that the law in (ii') converges to the law in (ii) as $\eps\to 0$. Now we will argue that, along the subsequence that was used to define $d_h^\beta$, the law in (i') also converges to the law in (i). Let $F$ be a continuous bounded functional on $\Hloc(\HH)$ and let $A\subset I$ be a Borel set. By uniform integrability of $d^\beta_{h,\ep}$, along the considered subsequence,
		\[ \frac{\P(d_{h,\eps}^\beta(A) F(h))}{\P(d_{h,\eps}^\beta(I))} \to \frac{ \P_\infty(d^\beta(A) F(h))}{\P_\infty(d^\beta(I))}=\frac{\P(d^\beta_h(A) F(h))}{\P(d_h^\beta(I))}, \]
		where we slightly abuse notation and also use $\P,\P_\infty$ to denote expectation relative to the probability measures $\P,\P_\infty$.
		Since the left-hand side is equal to the expectation of $F(h)\I_{\{z\in A\}}$ for $(h,z)$ sampled as in (i') and the right-hand side is equal to the same expectation for $(h,z)$ sampled as in (i), we can conclude that the law in (i') converges to the law in (i). Clearly the equivalence of (i') and (ii') for every $\eps$, together with the convergence (i') $\Rightarrow$ (i) and (ii') $\Rightarrow$ (ii) implies the equivalence of (i) and (ii).
	\end{proof}
	
	\begin{lemma}
		Let $(\wh h,z)$ have the law described in (i) of Lemma \ref{lem:rooted2}. Then as $C\rta\infty$ and conditioned on $z$, the surface $(\wh B(z,1),\wh h+C,z,z+i)$ converges in law to a $(2,2)$-quantum wedge.
		\label{prop:zoom0}
	\end{lemma}
	\begin{proof}
		One can check that the proof of Lemma \ref{prop:wedge-transform}  works  identically if we condition only on $z$ in (i) and (ii) rather than on $(z,\nu_h([a,z]),\nu_h([z,b])$. By this variant of Lemma \ref{prop:wedge-transform}, proving Lemma \ref{prop:zoom0} is equivalent to showing that conditioned on $z$, the quantum surface $(\HH,\wh h+C,z,\infty)$, with $\wh h$ viewed as a distribution on $\HH$, (i.e., we set $\wh h$ equal to 0 outside of $\wh B(z,1)$) converges in 
		law to a $(2,2)$-quantum wedge. Write $\wh B=\wh B(z,1)$ to simplify notation.
		Decompose $\wh h=\wh h_\cir+\wh h_\rad$, where $\wh h_\cir\in H_2(\wh B)$ and $\wh h_\rad\in H_1(\wh B)$. By the Markov 
		property, both the mixed GFF in $D$ and the Neumann GFF in $\HH$, when restricted to $\wh B$, can be written as the sum of a mixed GFF in $\wh B$ (with free boundary conditions on $\partial \wh B\cap\RR$ and zero boundary conditions on $\partial\wh B\setminus\RR$) plus a harmonic function that extends continuously to $\partial\wh B\cap\RR$. Therefore $\wh h_\cir$ and 
		$h_\cir^{\GFF}|_{\wh B}$ can be coupled 
		together so they differ by a random function which extends continuously to $\wh B\cap\RR$. In particular, $\wh h_\cir$ and 
		$h_\cir^{\GFF}|_{\wh B}$ can be coupled so that $\wh h_\cir(c\cdot)-h_\cir^{\GFF}|_{\wh B}(c\cdot)$ converges a.s.\ to a random constant as $c\rta 0$. It is therefore sufficient to show that if $h_\cir^{\GFF}$ is independent of $\wh h_\rad$ then $(\HH,h_\cir^{\GFF}+\wh h_\rad+C,z,\infty)$ converges in law to a $(2,2)$-quantum wedge as $C\rta\infty$.
		
		By Lemma \ref{lem:rooted2}, $\wh h_\rad$ can be coupled together with $A$ in (ii) of that lemma such that $\wh h_\rad(x)=A_{-\log|x-z|}$. Recall that $A$ can be coupled together with a 3-dimensional Bessel process $(\cB_s)_{s\geq 0}$ started from $-A_0+\beta$ such that $A_s=-\cB_{2s}+2s+\beta$. For $C>1$ define
		\eqbn
		\begin{split}
			&T^1_C = \inf\{ s\geq 0\,:\, \cB_{2s}=C+\beta \},\qquad
			T^3_C = \sup\{ s\geq 0\,:\, \cB_{2s}=C+\beta \},\\
			&T^2_C = \underset{s\in[T^1_C,T^3_C]}{\op{argmin}} \cB_{2s},\qquad\qquad\qquad\qquad
			\theta=\inf\{ \cB_{2s}\,:\,s\in[T^1_C,T^3_C] \}=\cB_{2T^s_C}.
		\end{split}
		\eqen
		Note that $(\cB_{t+2T^1_C})_{t\geq 0}$ has the law of a Bessel process started from $C+\beta$. By \cite[Theorem 3.5]{Wil74}, $\theta$ has the law of a uniform random variable on $[0,C+\beta]$, and, conditioned on $\theta$, 
		\begin{compactitem}
			\item[(i)] the process $(\cB_{s+2T^3_C}-(C+\beta) )_{s\geq 0}$ has the law of a Bessel process started from 0, and 
			\item[(ii)] $(\cB_{-s+2T^3_C} )_{s\in [0,2T^3_C-2T^2_C]}$ has the law of a Brownian motion started from $C+\beta$ and stopped at the first time it reaches $\theta$. 
		\end{compactitem}
		It follows that as $C\rta\infty$ the process $(\cB_{s+2T^3_C}-(C+\beta) )_{s\in\RR}$ converges in law to the \emph{negative} of the process considered in Remark \ref{rmk:wedge-strip} on any compact interval. Therefore $(-\cB_{s+2T^3_C}+(C+\beta)+2s)_{s\in\RR}$ converges in law to the process  $(h_{\rad}(e^{-s}))_{s\in\RR}$ in Definition \ref{def:wedge}, which concludes the proof.
	\end{proof}
	
	\begin{lemma} Assume the same set-up as in Lemma \ref{prop:zoom0}, except now without the assumption that $I$ is bounded away from $\HH\setminus D$ and the assumption \eqref{eq:Idist}. Then as $C\to \infty$ and conditioned on $(z,d_h^\beta([a,z]),d_h^\beta([z,b]))$, the surface $(\wh{B}(z,1),\wh{h}+C,z,z+i)$ converges in law to a $(2,2)$-quantum wedge (where $\wh h$ is identically equal to 0 on $\wh B(z,1)\setminus D$).
		\label{lem:zoom0more}
	\end{lemma}
	\begin{proof}
		First we will argue that $d_h^\beta$ is atomless a.s. Notice that $d_{h,\eps}^\beta(dz)\leq\nu_{h,\eps}(dz)+\beta \eps \e^{h_\eps(z)}dz$ (with equality on the event $\cC_\beta$; see Remark \ref{rmk::cb}). Since $\beta \eps \e^{h_\eps(z)}$ converges a.s.\ to 0 and $\nu_{h,\eps}(dz)$ converges a.s.\ to the non-atomic measure $\nu_h$ as $\ep\rta 0$, this implies that $d_h^\beta$ is atomless a.s.
		
		Now observe that the proof of Lemma \ref{prop:zoom0} above carries through just as before if we replace the $1$ on the right side of \eqref{eq:Idist} with some other constant $r\in(0,1)$. Then we see that Lemma \ref{prop:zoom0} also holds if $I$ is \emph{not} bounded away from $\HH\setminus D$, since any interval contained in $\partial D\cap\RR$ can be approximated arbitrarily well by an interval satisfying \eqref{eq:Idist} for some $r\in(0,1)$. This implies, since $d_h^\beta$ is atomless, that the point $z$ in the former case converges in total variation distance to the point $z$ in the latter case when $r\rta 0$. 
		
		From Lemma \ref{prop:zoom0} (without the assumption that $I$ is bounded away from $\HH\setminus D$), and by proceeding exactly as in the proof of \cite[Proposition 5.5]{Sh16}, we get that Lemma \ref{prop:zoom0} also holds if we condition on $d_h^\beta([a,z])$ and $d_h^\beta([z,b])$. 
	\end{proof}

	\begin{lemma}
		Let $(X_n,Y_n)$ for $n\in\N$ and $(X,Y)$ be random variables such that the vectors $(X_n,Y_n)$ converge in total variation distance to $(X,Y)$ as $n\rta\infty$. Assume $Y_n,Y$ are vectors in $\R^N$ for some $N\in\N$, while $X_n,X$ take values in some Borel space $(S,\cS)$.
		Then there exists a set $\mcl A\subset\R^N$ such that $\P[Y\in\mcl A]=1$, and such that for any $a\in\mcl A$ the law of $X_n$ given $Y_n=a$ converges to the law of $X$ given $Y=a$.\footnote{The conditional law of $X$ (resp.\ $X_n$) given $Y=a$ (resp.\ $Y_n=a$) exists for almost all $a$ sampled from the law of $Y$ (resp.\ $Y_n$) by e.g.\ \cite[Section 5.1.3]{durrett10}. Proceeding as in e.g.\ \cite[Exercise 33.16]{billingsley95} one can argue that  $\mathbb{P}(X\in \cdot \; | Y=a)=\lim_{\delta\to 0} \mathbb{P}(X\in \cdot \; | \|Y-a\|_\infty<\delta)$ for almost all $a$, and the same statement holds for $X_n,Y_n$ instead of $X,Y$.}
		\label{prop:cond}
	\end{lemma}
	\begin{proof}
		Let $\eps>0$. It is clearly sufficient to prove the lemma under the weaker requirement that $\mcl A$ satisfies $\P[Y\in\mcl A]\geq 1-(2\cdot 12^N+1)\ep$. For this, it suffices to show that for an arbitrary function $F:S\to\{0,1 \}$, any $a\in\mcl A$, and all sufficiently large $n$,
		\eqb
		\big|\P\big[ F(X)=1\,\big|\,Y=a \big]-\P\big[ F(X_n)=1\,\big|\,Y_n=a \big]|\leq 3\ep.
		\label{eq:cond0}
		\eqe
		Choose $n$ sufficiently large such that the total variation distance between $(X_n,Y_n)$ and $(X,Y)$ is smaller than $\eps^2$. We will work with such a fixed choice of $n$ in the remainder of the proof, and will prove that \eqref{eq:cond0} is satisfied. 
		
		Choose $\delta>0$ sufficiently small such that for all $a$ in a set $\mcl A_0\subset\R^N$ satisfying $\P[Y\in\mcl A_0]>1-\eps/2$ the following hold:
		\eqb
		\begin{split}
			&\big|\P\big[ F(X)=1\,\big|\,Y=a \big] - 
			\P\big[ F(X)=1\,\big|\,\|Y-a\|_\infty<\delta \big]\big|<\eps;\\
			&\big|\P\big[ F(X_n)=1\,\big|\,Y_n=a \big] - 
			\P\big[ F(X_n)=1\,\big|\,\|Y_n-a\|_\infty<\delta \big] \big|<\eps.
		\end{split}
		\label{eq:cond}
		\eqe Let $K\subset\mathcal{A}_0$ be a compact set such that $\P[Y\in K]>1-\eps$. For any $a\in\R^N$ define $\mcl N(a)=\{z\in\R^N\,:\,\|z-a\|_\infty\leq\delta \}$. Say that a point $a\in K$ is \emph{bad} if $\P[Y\in\mcl N(a)]=0$, $\P[Y_n\in\mcl N(a)]=0$, or the total variation distance between $(X_n,Y_n)$ and $(X,Y)$ conditioned on $Y_n\in \mcl N(a)$ and $Y\in \mcl N(a)$, respectively, is at least $\eps$. A point in $K$ which is not bad is \emph{good}. Let $\mcl B\subset K$ denote the set of bad points. We will prove that 
		\eqb
		\P[Y\in\mcl B] \leq 2\cdot 12^N\eps.
		\label{eq:bad}
		\eqe	
		Taking $\mathcal{A}=K\setminus \mcl B$ and applying \eqref{eq:cond} then completes the proof.
		
		Choose points $a_1,\dots,a_M\in \mcl B$ for some $M\in\N$ using the following rule. Given $a_1,\dots,a_m$ let $a_{m+1}\in \mcl B$ be chosen such that $\mcl N(a_{m+1})$ is disjoint from $\mcl N(a_1),\dots,\mcl N(a_m)$, and such that $\P[Y\in\mcl N(a_{m+1})]$ is maximized. Let $M$ be the smallest $m$ such that there is no possible way to choose $a_{m+1}$ (i.e., all points in $\mcl B$ have $\|\cdot\|_\infty$ distance less than $\delta$ from $\mcl N(a_1)\cup\dots\cup\mcl N(a_m)$). Define $\frk m=\P[Y\in\mcl N(a_1)]+\dots+\P[Y\in\mcl N(a_M)]$. 
		
		The idea for the proof of \eqref{eq:bad} is that $\frk m$ has to be small because the total variation distance between $(X_n,Y_n)$ and $(X,Y)$ is assumed to be small, and that by the definition of the $\{\mathcal{N}(a_i)\}_i$, $\mathbb{P}(Y\in \mathcal{B})$ is of order $O(\mathfrak{m})$.
		
		Proceeding with the details, since $\mcl N(a_1),\dots,\mcl N(a_M)$ are disjoint, we can bound the total variation distance between $(X_n,Y_n)$ and $(X,Y)$ from below by summing the contribution from each set $\mcl N(a_m)$. More precisely, for arbitrary Borel (not necessarily probability) measures $\sigma_1,\sigma_2$ defined on $\R^N$, define
		$$
		d_{\op{tv}}(\sigma_1,\sigma_2)= \sup_{A\subset\R^N} |\sigma_1(A)-\sigma_2(A)|,
		$$
		and note that this defines a metric on the set of Borel measures on $\R^N$. 
		Let $\sigma_n$ (resp.\ $\sigma$) denote the law of $(X_n,Y_n)$ (resp.\ $(X,Y)$), and let $\sigma_n^m=\sigma_n|_{\mcl N(a_m)}$ and $\sigma^m=\sigma|_{\mcl N(a_m)}$.
		For an arbitrary measure $\wh\sigma$ let $|\wh\sigma|$ denote its total mass.
		By the triangle inequality and since 
		$d_{\op{tv}}(\sigma^m_n,\frac{|\sigma^m|}{|\sigma^m_n|}\sigma^m_n)=\big||\sigma^m_n|-|\sigma^m|\big|\leq d_{\op{tv}}(\sigma^m,\sigma^m_n)$,
		$$
		d_{\op{tv}}\Big(\frac{\sigma^m}{|\sigma^m|},\frac{\sigma^m_n}{|\sigma^m_n|}\Big)
		\leq 
		\frac{1}{|\sigma^m|} d_{\op{tv}}(\sigma^m,\sigma^m_n)
		+ \frac{1}{|\sigma^m|} d_{\op{tv}}\Big(\sigma^m_n,\frac{|\sigma^m|}{|\sigma^m_n|}\sigma^m_n\Big)
		\leq \frac{2}{|\sigma^m|} d_{\op{tv}}(\sigma^m,\sigma^m_n).
		$$
		Using $\frk m=\sum_{m=1}^{M}|\sigma^m|$ we now get
		\eqbn
		\frac 12 \frk m \eps
		\leq \sum_{m=1}^M \frac 12 |\sigma^m| \cdot d_{\op{tv}}\Big( \frac{\sigma^m}{|\sigma^m|}, \frac{\sigma_n^m}{|\sigma_n^m|} \Big)
		\leq \sum_{m=1}^M d_{\op{tv}}( \sigma^m,\sigma_n^m )
		\leq d_{\op{tv}}( \sigma,\sigma_n )
		< \eps^2,
		\eqen
		which gives $\frk m\leq 2\eps$. Using this, we get \eqref{eq:bad} if we can prove the following
		\eqb
		\P[Y\in\mcl B]\leq 12^N\frk m.
		\label{eq:bad1}
		\eqe
		Let $\mcl M(a_m)$ be the box of side length $6\delta$ centred at $a_m$, minus the union of $\mcl M(a_1),\dots\mcl M(a_{m-1})$. To prove \eqref{eq:bad1} it is sufficient to show that: (i) $\P[Y\in\mcl B\cap\mcl M(a_m)]\leq 12^N\P[Y\in\mcl N(a_m)]$, and (ii) $\mcl B\subset \mcl M(a_1)\cup\dots\cup \mcl M(a_M)$, since then we have
		$$
		\P[ Y\in\mcl B ]
		\leq \sum_{m=1}^{M} \P[Y\in \mcl B\cap\mcl M(a_m) ]
		\leq \sum_{m=1}^{M} 12^N \P[ Y\in \mcl N(a_m) ]
		= 12^N \frk m.
		$$
		Assertion (i) follows upon dividing the box of side length $6\delta$ centred at $a_m$ %$\{z\in\R^N\,:\,\|z-a_m\|_\infty\leq 3\delta \}$ 
		(which contains $\mcl M(a_m)$) into $12^N$ boxes $b$ of side length $\delta/2$, and using that, by the definition of $a_m$, if $b\cap \mcl M(a_m)\cap\mcl B\neq\emptyset$ then $\P[Y\in b]\leq \P[Y\in\mcl N(a_m)]$. Assertion (ii) follows by using the definition of $M$ and since (by the definition of $\mcl M(a_1),\dots,\mcl M(a_M)$) points in the complement of $\mcl M(a_1)\cup\dots\cup \mcl M(a_M)$ have distance at least $2\delta$ from $\mcl N(a_1)\cup\dots\cup \mcl N(a_M)$. We conclude that \eqref{eq:bad} and \eqref{eq:bad1} both hold.
	\end{proof}

	\begin{proof2}{Proposition \ref{prop:zoom}}
		Consider the law on $(z,h)$ that can be sampled from as follows. First sample $h$ from $\P$ (i.e., as in Proposition \ref{prop:zoom}). Then, on the event $d_h^\beta(I)>0$ sample $z$ from $d_h^\beta$ normalised to be a probability measure, and otherwise sample $z$ from Lebesgue measure on $I$.
		
		Write $\wh h$ for $h$ restricted to $\wh B(z,1)$. Observe that on the event $\{d^\beta_h(I)>0\}$, the conditional law of $\wh h$ given $(z,d_h^\beta([a,z]),d_h^\beta([z,b]))$ is exactly the same as the conditional law of $\wh h$ given $(z,d_h^\beta([a,z]),d_h^\beta([z,b]))$ in Lemma \ref{lem:zoom0more}. This is because we have conditioned on the Radon--Nikodym derivative between the two different laws on $h$. Hence, under the law on $(z,h)$ just defined, on the event that $d_h^\beta(I)>0$ and conditionally on $(z,d_h^\beta([a,z]),d_h^\beta([z,b]))$, \[(\wh{B}(z,1),\wh{h}+C,z,z+i) \text{ converges in law to a } (2,2) \text{-quantum wedge as } C\to \infty.\]  
		Finally, by Lemma \ref{lem::d-nu}, Remark \ref{rmk::dbetanu}, and Lemma \ref{prop:cond}, letting $\beta\to \infty$, we can conclude that if we sample $h$ from $\P$, then sample $z$ from $\nu_h$, and let $\wh h$ be $h$ restricted to $\wh B(z,1)$, then we have that conditionally on $(z,\nu_h([a,z]),\nu_h([z,b]))$, $(\wh{B}(z,1),\wh{h}+C,z,z+i)$ converges in law to a $(2,2)$-quantum wedge as $C\to \infty$.  
		Proposition \ref{prop:zoom} now follows upon application of Lemma \ref{prop:wedge-transform}.
	\end{proof2}

	\section{The critical quantum zipper via subcritical approximation}
	\label{sec:approx}
	
	The goal of this section is to prove Proposition \ref{lem::main_criticalzipperconv} below. 
	In this proposition it is shown that given a $(2,1)$-quantum wedge $(\HH,h,0,\infty)$, one can conformally weld together the intervals to the left and right of the origin with quantum boundary length one.
	Concretely, it provides the existence of a conformal map $f$ (measurable with respect to $h$) from $\HH$ to $\HH\setminus \wt \eta$, where $\wt \eta$ is a section of a simple curve starting from $0$, such that any two points to the left and right of $0$ with equal $\nu_h$ boundary length (less than one) are mapped to the same point on $\wt \eta$. Moreover, if one also starts with a curve $\eta$ on $(\HH,h,0,\infty)$, that is independent of $h$ and has the law of an $\SLE_4$ from 0 to $\infty$, then the new field/curve pair defined by $f(h)$ and $\wt \eta \cup f(\eta)$ has the same law as $(h,\eta)$. 
	
	The strategy is to use the fact that such an operation exists \cite{Sh16} in the subcritical case $\gamma\in(0,2)$, i.e., when the $\SLE_4$ is replaced by an $\SLE_{\gamma^2}$, the $(2,1)$-quantum wedge is replaced by a $(\gamma,\gamma-2/\gamma)$-quantum wedge, and critical boundary length is replaced by $\gamma$-LQG boundary length. See Figure \ref{fig:approx} for an illustration.
	We will show that a number of limits can be taken as $\gamma\uparrow 2$, using, for example, the fact
	that critical LQG measures can be obtained as a limit of subcritical measures (Lemma \ref{lem::cd_critwedge}). Combining these convergence statements provides the existence of the welding operation. Making use of Theorem \ref{thm2}, we can prove that the conformal map $f$ is measurable with respect to $\eta$. 
	
	\begin{figure}
		\centering
		\includegraphics[scale=1]{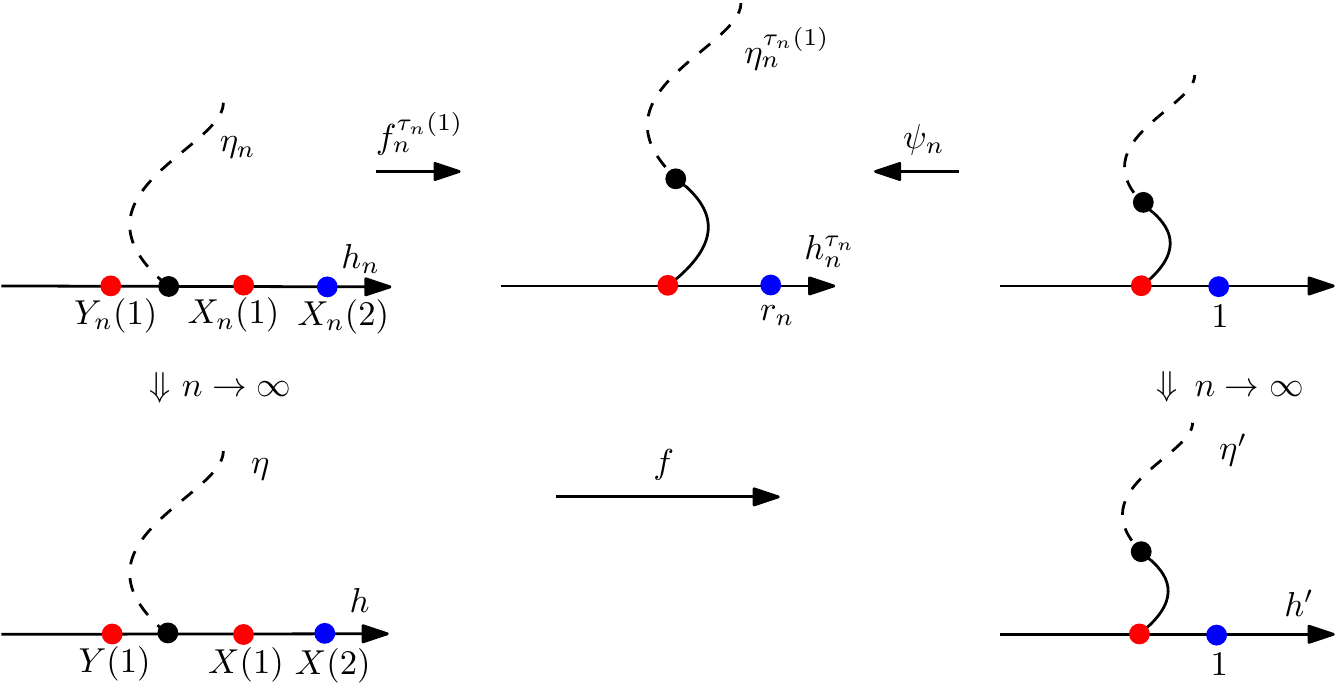}
		\caption{Illustration of objects defined in Section \ref{sec:criticalzipperapprox}. Our strategy is to construct the critical quantum zipper (lower row) by taking the $n\rta\infty$ limit of the subcritical quantum zipper (upper row). The convergence in law indicated by the two vertical arrows is joint as $n\rta\infty$. }
		\label{fig:approx}
	\end{figure}
	
	Let $(\gamma_n)_{n\in\N}$ be a sequence in $(0,2)$ with $\gamma_n\uparrow 2$ as $n\to \infty$, and set $\kappa_n=\gamma_n^2$. We will always denote by $h_n$ a random element of \en{$\Hloc(\HH)$} with the law of a $(\gamma_n,\gamma_n-2/\gamma_n)$-quantum wedge, as in Definition \ref{def:wedge}. That is, $(\HH,h_n,0,\infty)$ has the distribution of the equivalence class representative of a $(\gamma_n, \gamma_n-2/\gamma_n)$-quantum wedge in the last exit parameterisation. Given such an $h_n$ we denote by $\nu_n=(4-2\gamma_n)^{-1}\nu_{h_n}^{\gamma_n}$, the associated (renormalised) $\gamma_n$-Liouville boundary measure on $\RR$ (as in Lemma \ref{lem::defnlmqw}). For $q\in \mathcal{Q}:=\mathbb{Q}\cap [0,1]$, we denote 
	\eqb \label{eqn:XY} X_n(q)=\inf \{x\ge 0:\, \nu_n([0,x])\ge q\}; \;\; Y_n(q)=-\inf\{y\ge 0: \nu_n([-y,0])\ge q\}; \eqe so that $(X_n(q),Y_n(q))$ is a pair of points to the right and left, respectively, of $0$, with $\nu_n([Y_n(q),0])=\nu_n([0,X_n(q)])=q$. 
	
	Similarly, $h$ will always \en{denote an element of $\Hloc(\HH)$} with the law of a $(2,1)$-quantum wedge (in the last exit parameterisation) and $\nu_h=:\nu$  will be the critical boundary measure associated to $h$ (as in Lemma \ref{lem::defnlmqw}). For $q\in \mathcal{Q}$ we define $(X(q),Y(q))$ corresponding to $\nu$ as in \eqref{eqn:XY}, so that $(X(q),Y(q))\in [0,\infty)\times (-\infty,0]$ and $\nu([0,X(q)])=\nu([Y(q),0])=q$. 
	
	\begin{lemma}\label{lem::wedge_meas_convergence}
		There exists a coupling of $((h_n)_{n\in \N},h)$ such that a.s.\ as $n\rta\infty$,
		\[ (h_n, \nu_n)\to (h,\nu).\]
		This is with respect to the topology of \en{$\Hloc(\HH)$} in the first coordinate, and the local weak topology for measures on $\RR$ in the second.
	\end{lemma}
	
	\begin{proof}
		By the Skorokhod representation theorem, it is sufficient to show that $(h_n,\nu_n)$ converges in distribution to $(h,\nu)$ as $n\to \infty$. The idea is that away from $0$ and the unit circle, $h_n$ is arbitrarily close to $h$ in total variation distance for large $n$, so we can essentially just apply Lemma \ref{lem::cd_critwedge} in these regions. We then deal with neighbourhoods of $0$ and the unit circle separately; showing that as the size of the neighbourhoods goes to $0$ the behaviour of $(h_n,\nu_n)$ restricted to these neighbourhoods can be neglected (uniformly in $n$).
		
		To carry out this idea, when $x\in \R$ and $r>s>0$ we write
		$$
		\wh B(x,r):=\{w\in \HH: |w-x|< r\} \text{ and } \wh A(x,r,s):=\wh B(x,r)\setminus \wh B(x,s).
		$$
		First, we observe that for any $(r_i)_{i=1}^4$ such that $r_4>r_3>1>r_2>r_1>0$ there exists a sequence of couplings $(h_n,h)$, such that $\mathbb{P}(h_n=h \text{ on } \wh A(0,r_4,r_3) \cup \wh A(0,r_2,r_1))$ tends to $1$ as $n\to \infty$. Indeed, since we can couple the fields to have the same circular part, this just follows because setting:
		\begin{compactitem}
			\item $\cL_n$ to be the law of a double sided Brownian motion plus drift $(2-2/\gamma_n)$, restricted to some interval $[-M,M]$ and conditioned to stay below the curve $s\mapsto Q_{\gamma_n}s$ for all positive time; and 
			\item $\cL$ to be the law of a double sided Brownian motion with drift $1$, restricted to $[-M,M]$ and conditioned to stay below the curve $s\mapsto 2s$ for all positive time,
		\end{compactitem} then $\cL_n\to \cL$ with respect to total variation distance as $n\to \infty$. 
		Hence, by Lemma \ref{lem::cd_critwedge}, %and since convergence in $L^2$ implies convergence in $\Hloc$, it
		it suffices to show that 
		\eqb  \nu\big([-\delta,+\delta]\cup[-1+\delta,-1-\delta]\cup[1+\delta,1-\delta]\big)\, \to 0; \;\; \|h \|_{\en{\Hinv}(\wh B(0,\delta))}\to 0; \;\; \|h\|_{\en{\Hinv}(\wh A(0,1+\delta, 1-\delta))} \to 0 \label{eqn:mc1}
		\eqe in probability (equivalently, in distribution) as  $\delta\to 0$, and 
		\eqb \mathbb{P}\big( \nu_n([-\delta,\delta]\, \cup\, [-1+\delta,-1-\delta]\, \cup\, [1+\delta,1-\delta])>\eta\big) \to 0 \label{eqn:mc2} \eqe
		\eqb   \mathbb{P}\big( \|h_n \|_{\en{\Hinv}(\wh B(0,\delta))}\ > \eta\big) \to 0; \;\;   \mathbb{P}\big( \|h_n \|_{\en{\Hinv}(\wh A(0,1+\delta,1-\delta))} > \eta\big) \to 0
		\label{eqn:mc3}\eqe
		for any $\eta>0$, uniformly in $n$, as $\delta\to 0$.
		
		The first statement of \eqref{eqn:mc1}  holds because $\nu$ is a.s.\ atomless (Remark \ref{rmk:propslm}). Moreover, \eqref{eqn:mc3} and the last two statements of \eqref{eqn:mc1} follow by
		decomposing $h$ and $(h_n)_{n\in \N}$ into their projections onto $H_1(\HH)$ and $H_2(\HH)$. Indeed, the projections onto $H_2(\HH)$ all have the same law -- that of $h_{\cir}^{\GFF}$ -- and it can be verified by a direct computation that the $\en{\Hinv}(\wh B(0,\delta)\cup \wh A(0,1+\delta, 1-\delta))$ norm of $h^{\GFF}_{\cir}$ goes to $0$ in probability as $\delta\to 0$. The projections onto $H_1(\HH)$, when restricted to $\wh B(0,1+\delta)$, can also all be stochastically dominated (for example) by the random function $$\I_{\{z\in \wh A(0,1+\delta, 1)\}} B_{2\log|z|}-2\log(|z|),$$ 
		where $B$ is a standard Brownian motion. \en{One can easily check that this function has $L^2(\wh B(0,\delta)\cup \wh A(0,1+\delta, 1-\delta))$ norm going to $0$ in probability as $\delta\to 0$, which is more than we need.}
		
		For \eqref{eqn:mc2}, first fix $\eta>0$. We will deal with the neighbourhood $[-\delta,\delta]$ of $0$, and the intervals $[\pm 1 -\delta,\pm 1+\delta]$ around $\pm 1$, separately. To show that $\P(\nu_n([-\delta,\delta]>\eta)\to 0$ uniformly in $n$ as $\delta\to 0$, we observe (as in the proof of Lemma \ref{lem::defnlmqw}) that if $\wh h$ is a Neumann GFF with additive constant fixed so that its average on $\partial \wh B(0,1)$ is equal to $0$, then for every $n$ there exists a random constant $c_n$ such that \eqbn c_n^{-\gamma_n^2/2}\nu_{\wh h+(\gamma_n-2/\gamma_n)\log|\cdot|}^{\gamma_n}([-c_n\delta,c_n \delta]) \eqD (4-2\gamma_n)\nu_n([-\delta,\delta]).\eqen 
		Moreover, the probability that $c_n$ is greater than $M$ goes to $0$ uniformly in $n$ as $M\to \infty$, since, by the proof of Lemma \ref{prop:strip}, $c_n$ has the law of the exponential of minus the last time that a Brownian motion with negative drift $B_{2t}-(Q_{\gamma_n}-\gamma_n+2/\gamma_n)t$ is greater than or equal to $0$. Hence it suffices to show that \eqb \label{eqn:41} (4-2\gamma_n)^{-1}\int_{-\delta}^{\delta} |z|^{-(\gamma_n^2-4)} \nu_{\wh h}^{\gamma_n}(dz) \to 0 \eqe in probability (or, equivalently, in distribution) as $n\to \infty$. However, it follows from \cite[Lemmas 3.1 and 3.2]{APS18two} (with straightforward adaptation to the boundary case) that the integral in \eqref{eqn:41}
		has $(1-\gamma_n/2)^{\text{th}}$ moment converging to $0$ with $\delta$, uniformly in $n$. Since $(4-2\gamma_n)^{(1-\gamma_n/2)}\to 1$ as $n\to \infty$, the result then follows by Markov's inequality.
		
		To show that $\P(\nu_n([\pm 1-\delta,\pm 1 +\delta]>\eta)\to 0$ uniformly in $n$ as $\delta\to 0$, we first note that (by Lemma \ref{lem::cd_critwedge}) this would hold if the fields $h_n$ were all replaced by a Neumann GFF in $\HH$, with additive constant fixed so that its average on $\partial \wh B(0,1)$ is $0$. Then, since 
		\begin{compactitem}
			\item such a Neumann GFF can be written as the sum of $h_\cir^{\GFF}$ plus a random function whose supremum in $\wh B(\pm 1, 2\delta)$ goes to $0$ as $\delta\to 0$, and 
			\item $h_n$ can be written as the sum $h_\cir^{\GFF}+F_n$ where $\P\big(\sup_{\wh B(\pm 1, 2\delta)} F_n \ge a\big)\to 0$ as $\delta\to 0$ uniformly in $n$ for any fixed $a>0$,
		\end{compactitem} the result follows.
	\end{proof}
	
	\begin{lemma}\label{lem::wedge_boundarypoints_conv}
		There exists a coupling of $((h_n,\eta_n)_{n\in \N}, h,\eta)$ such that:
		\begin{itemize}
			\item $(h_n,\eta_n)$ for each $n$ has the marginal law of a $(\gamma_n,\gamma_n-2/\gamma_n)$-quantum wedge and an independent $\SLE_{\kappa_n}$ from $0$ to $\infty$ in $\HH$ ($\kappa_n=\gamma_n^2)$;
			\item $(h,\eta)$ has the marginal law of a $(2,1)$-quantum wedge and an independent $\SLE_4$ from $0$ to $\infty$ in $\HH$;
			\item $(h_n,\eta_n,(X_n(q))_{q\in \mathcal{Q}}, (Y_n(q))_{q\in \mathcal{Q}})$ converges to $(h,\eta, (X(q))_{\qiq}, (Y(q))_{\qiq})$ in probability as $n\to \infty$, with respect to \en{$\Hloc(\HH)$} convergence in the first coordinate, Carath\'{e}odory convergence in the second coordinate, and the product topology on $\RR^{\mathcal{Q}}$ in the third and fourth coordinates.\footnote{Here $(X_n(q),Y_n(q))_{n\in \N, q\in \mathcal{Q}}, (X(q),Y(q))_{q\in \mathcal{Q}}$ are defined with respect to $((h_n)_{n\in \N}, h)$ as in the introduction to this section.}
		\end{itemize}
	\end{lemma}
	
	\begin{proof}
		First, by Lemma \ref{lem:eta_conv}, it is possible to couple a sequence of $\SLE_{\kappa_n}$ curves and an $\SLE_{4}$ such that one has convergence with respect to the Carath\'{e}odory topology in probability as $n\to \infty$.  Next, since the curves can be sampled independently of everything else in the statement, it is enough to show that with the coupling of Lemma \ref{lem::wedge_meas_convergence}, we have $(X_n(q))_{q\in \mathcal{Q}}$ converging to $(X(q))_{\qiq}$ and $(Y_n(q))_{\qiq}$ converging to $(Y(q))_{\qiq}$ in probability as $n\to \infty$ (with respect to the product topology on $\RR^{\mathcal{Q}}$). We will show the statement for $X$; the corresponding statement for $Y$ follows by the same argument. 
		
		Let $F_{n}:[0,\infty)\to[0,\infty)$ and $F:[0,\infty)\to[0,\infty)$ describe the cumulative mass of the measures $\nu_{n}$ and $\nu$, i.e., $F_{n}(x)=\nu_{n}([0,x])$ and $F(x)=\nu([0,x])$ for all $x$, and note that a.s.\, by Remark \ref{rmk:propslm}, both are continuous and strictly increasing. 
		This means that $F_{n}$ converges pointwise to $F$ a.s.\ as $n\to \infty$, and hence also that the generalised inverses 
		\eqb F_{n}^{-1}(s)=\inf\{x\in[0,\infty)\,:\, F_{n}(x)\geq s \}\eqe converge pointwise to the generalised inverse $F^{-1}$ (defined analogously) a.s.\ as $k\to \infty$. In particular, this implies that $(X_{n}(q))_{\qiq}$ converges to $(X(q))_{\qiq}$  a.s.\ as $n\to \infty$, with respect to the product topology on $\RR^{\mathcal{Q}}$. 
	\end{proof}
	\\
	
	For what follows, we need to recall the definition of Sheffield's \emph{capacity quantum zipper} \cite{Sh16} for $\gamma\in (0,2)$. 
	\begin{defn}
		\label{def:cqz}
		Let $\gamma\in (0,2)$ and $(\HH,\wh h_0,0,\infty)$ be an equivalence class representative of a $(\gamma,\gamma-2/\gamma)$-quantum wedge. Let $\kappa=\gamma^2$, and let $\wh \eta_0$ be an independent $\SLE_{\kappa}$ in $\HH$ from $0$ to $\infty$. 
		Then the \emph{capacity quantum zipper} is a centered, reverse Loewner flow $(\wh f_t)_{t\ge 0}$ coupled with $(\wh h_0,\wh \eta_0)$, such that:
		\begin{compactitem}
			\item $(\wh f_t)_{t\ge 0}$ is measurable with respect to $\wh h_0$;
			\item the marginal law of $(\wh f_t)_{t\ge 0}$ is a centered, reverse $\SLE_{\kappa}$ flow parameterised by half-plane capacity;
			\item for any $t$ and $x\in \wh \eta_t \setminus \wh f_t(\wh \eta_0)$, denoting by $\eta_x^\L$ and $\eta_x^\Ro$ the left- and right-hand sides of $\eta$ up to $x$, the $\nu^{\gamma}_{\wh h_0}$ length of the intervals $\wh f_t^{-1}(\eta_x^\L)$ and $\wh f_t^{-1}(\eta_x^\Ro)$ agree.
		\end{compactitem}
		This induces a dynamic 
		$$(\wh h_{t}, \wh \eta_{t}):=(\wh h_0 \circ \wh f_{t}^{-1}+Q_\gamma \log |(\wh f_{t}^{-1})'|, \wh f_{t}(\wh \eta_0))$$ on $(\wh h_0, \wh \eta_0)$ which is stationary when observed at \emph{quantum typical times}.	More precisely, for any $l\ge0$, if $$X_l=\inf\{x\ge 0: \nu_{h_0}^\gamma(0,x)=l\} \text{ and } T_l=\inf\{t\ge 0: f_t(X_l)=0\},$$ then $(\wh h_{T_l}, \wh \eta_{T_l})$ is equal in distribution, \emph{as a quantum surface}, to $(\wh h_0, \wh \eta_0)$.\footnote{By this we mean that $(\HH,\wh h_t,0,\infty)$, up to the equivalence described in Definition \ref{def::lqg}, is a $(\gamma,\gamma-2/\gamma)$-quantum wedge, and $\wh \eta_t$ is an $\SLE_\kappa$ that is independent of $\wh h_t$.}
	\end{defn}
	
	This flow thus represents a dynamic welding of $[0,\infty)$ to $(-\infty,0]$, according to the $\gamma$-LQG boundary length. It is essentially the same as the dynamic defined in the (subcritical version of) Theorem \ref{thm:criticalzipper}, but with a different time parameterisation.
	
	Now, assume that $((h_n,\eta_n)_{n\in \N}, h,\eta)$ are coupled together as in Lemma \ref{lem::wedge_boundarypoints_conv} and that $(X_n(q),Y_n(q))_{n\in \N,\qiq}$ and  $(X(q),Y(q))_{\qiq}$ are defined as in \eqref{eqn:XY} with respect to $(h_n)_{n\in \N}$ and $h$, respectively.
	For each $n\in \N$, let $(f_n^t)_{t\ge 0}$ be the centered reverse flow in Definition \ref{def:cqz}, when $(\wh h_0,\wh \eta_0)$ are replaced by $(h_n,\eta_n)$. For $q\in \mathcal{Q}$ we let $\tau_n(q)$ be the time at which $(X_n(q),Y_n(q))$ are mapped to $0$ by $f_n$. For $t\ge 0$, let  $h^t_n=f_n^{t}(h_n):=h_n\circ (f_n^{t})^{-1}+ Q_{\gamma_n}\log|((f_n^{t})^{-1})'|$ and $\eta_n^{t}=f_n^t(\eta_n)$. As in footnote \ref{footnote:offcurve}, although $h_n^t$ is only defined on the slit domain $\HH\setminus f_n^t(\eta_n)$ we can view it as an element of \en{$\Hloc(\HH)$}. Then by the properties described in Definition \ref{def:cqz}, it follows that for any $q\in \mathcal{Q}$:
	\begin{compactitem}
		\item $\eta_n^{\tau_n(q)}$ and $h_n^{\tau_n(q)}$ are independent; 
		\item $\eta_n^{\tau_n(q)}$ has the law of an $\SLE_{\kappa_n}$ from $0$ to $\infty$; and 
		\item $h_n^{\tau_n(q)}$ is (equivalent as a doubly-marked $\gamma_n$-quantum surface to) a $(\gamma_n,\gamma_n-2/\gamma_n)$-quantum wedge. 
	\end{compactitem} 
	We also define $r_n:=f_n^{\tau_n}(X_n(2))$ for each $n$, where the definition of $X_n(q)$ for $q\in \mathcal{Q}$ is extended in the obvious way to $X_n(2)$. Let $\psi_n$ denote the scaling map $z\mapsto r_n z$ on $\HH$.
	
	This next proposition provides, by approximation, the existence of a local conformal welding of $[0,\infty)$ to $(-\infty,0]$ for a $(2,1)$-quantum wedge. 
	\begin{propn}\label{lem::main_criticalzipperconv}
		Suppose that $((h_n,\eta_n)_{n\in \N}, h,\eta)$ are coupled together as in Lemma \ref{lem::wedge_boundarypoints_conv}, on some probability space $(\Omega, \mathcal{F},\mathbb{P})$. Then there exists a conformal map $f:\HH\to\HH$ and a pair $(h', \eta')$ with 
		\begin{equation}\label{eqn::main_conv_gamma2} (h_n,\eta_n,\psi_n^{-1}\circ f_n^{\tau_n(1)}, h_n^{\tau_n(1)}\circ \psi_n+Q_{\gamma_n}\log|r_n|, \psi_n^{-1}(\eta_n^{\tau_n(1)}))\overset{\mathbb{P}}{\longrightarrow} (h,\eta,f,h',\eta') 
		\end{equation}
		as $n\to \infty$ where the convergence is with respect to \en{$\Hloc(\HH)$} in the first and fourth coordinates, with respect to Carath\'{e}odory convergence in the second and fifth coordinates, and with respect to uniform convergence on $\{\HH+\im \eps\}$ for every $\eps>0$ in the third coordinate. Furthermore, we have that: 
		\begin{compactitem}
			\item[(a)] $(\HH, h',0,\infty)$ (viewed as a doubly-marked 2-quantum surface) has the law of a $(2,1)$-quantum wedge, and $\nu_{h'}([0,1])=1$;
			\item[(b)] $\eta'$ has the law of an $\SLE_4$ from $0$ to $\infty$ in $\HH$; 
			\item[(c)] $h'$ and $\eta'$ are independent; 
			\item[(d)] $\eta'=f(\eta)$ and $h'=f(h):=h\circ f^{-1}+2\log |(f^{-1})'|$ a.s.\;
			\item[(e)] $f(X(1))=f(Y(1))=0$, and $f(X(q))=f(Y(q))$ for every $q\in \mathcal{Q}$ a.s.\;
			and finally
			\item[(f)] $(f,h',\eta')$ is measurable with respect to $\sigma(\{h,\eta\})$. 
		\end{compactitem}
		\label{prop:joint-conv}
	\end{propn}
	
	The set-up for the proof is as follows. 
	Consider the joint law of the tuple, for $n\in \N$:
	\begin{equation}\label{eqn::hugetuple} \left(h_n,\eta_n,(X_n(q))_{q\in \mathcal{Q}}, (Y_n(q))_{\qiq}, (f_n^t)_{t\ge 0},\tau_n(1),h_n^{\tau_n(1)}\circ \psi_n+Q_{\gamma_n}\log|r_n|, \psi_n^{-1}(\eta_n^{\tau_n(1)}),r_n\right).\end{equation} 
	We consider the topology of \en{$\Hloc(\HH)$} in the 1st and 7th coordinates, Carath\'{e}odory convergence in the 2nd and 8th coordinates, pointwise convergence (i.e., with respect to product topology on $\RR^{\mathcal{Q}}$) in the 3rd and 4th coordinates, convergence on $\RR$ in the 6th and 9th coordinates, and Carath\'{e}odory+ convergence in the 5th coordinate.
	
	\begin{lemma}\label{claim::tightness}
		With respect to product topology above, the tuple \eqref{eqn::hugetuple} is tight. Furthermore, if $(h,\eta, X,Y,(f^t)_{t\ge 0}, \tau, h',\eta',r)$ denotes a subsequential limit, then (a) and (b) of Lemma  \ref{prop:joint-conv} are satisfied and $r'$ is a.s.\ strictly positive. 
	\end{lemma}
	\begin{proof}
		First observe that by Lemma \ref{lem::wedge_boundarypoints_conv} we have joint convergence in distribution of the first four coordinates. It remains to prove tightness of the remaining five coordinates, and verify the asserted properties of the subsequential limit. The sequence $((f_n^t)_{t\ge 0})_{n\in \N}$ is tight, since by Lemma \ref{lem::df_conv} we have convergence in distribution to a reverse $\SLE_4$ with respect to the Carath\'{e}odory+ topology. It is also immediate that the sequence $(\eta_n^{\tau_n(1)})_{n\in \N}$ is tight and that $\eta'$ satisfies (b), since by stationarity of the subcritical quantum zipper, the law of $\eta_n^{\tau_n(1)}$ is that of an $\SLE_{\kappa_n}$ curve from $0$ to $\infty$ in $\HH$, and the map $\psi_n$ is independent of $\eta_n^{\tau_n(1)}$.
		
		To see that the sequence $(\tau_n(1))_{n\in \N}$ is tight, first observe that $\mathbb{P}(X_n(1)\ge M)\to 0$ as $M\to \infty$, uniformly in $n$ (since we already know that $X_n(1)$ converges in probability). Therefore, we need only show that for fixed $M$, if $\sigma_M^n$ is the first time $M$ hits $0$ under a reverse $\SLE_{\kappa_n}$ flow, then $\mathbb{P}(\sigma_n(M)>K)\to 0$ as $K\to \infty$ uniformly in $n$. This follows directly from the proof of Lemma \ref{lem::df_conv}.
		
		Finally, by stationarity of the subcritical quantum zipper and definition of $r_n$ we know that for every $n$: \begin{compactitem}\item $(\HH,h_n^{\tau_n(1)},0,\infty)$ is equal in law to a $(\gamma_n,\gamma_n-2/\gamma_n)$-quantum wedge (when viewed as a quantum surface); and  
			\item $(4-2\gamma_n)^{-1}\nu_{h^{\tau_n(1)}_n}^{\gamma_n}([0,r_n]) =1$.
		\end{compactitem}
		It therefore follows from Lemma \ref{lem::wedge_meas_convergence} that $ h_n^{\tau_n(1)}\circ \psi_n+Q_n \log |r_n|$ converges in distribution to the equivalence class representative of a $(2,1)$-quantum wedge in $\HH$, with marked points at $0$ and $\infty$, that gives critical boundary length $1$ to the interval $[0,1]$. In particular, (a) holds and $(h_n^{\tau_n(1)}\circ \psi_n+Q_n \log |r_n|)_{n\in\N}$ is tight. 
		
		To prove tightness of $(r_n)_{n\in \N}$, and the assertion about positivity of any subsequential limit, we will show that
		\eqb \label{eqn:rn} 
		\mathbb{P}(r_n\notin[1/M,M])\to 0 \text{ as } M\to \infty, \text{ uniformly in } n. 
		\eqe
		For this we use the fact \cite[Proposition 4.11]{BN11} that $|X_n(2)-X_n(1)|\le r_n$, and if $\xi_n$ is the driving function of $f_n$, then $X_n(2)=|(f_n^{\tau_n(1)})^{-1}(r_n)|\ge r_n-|\xi_n^{\tau_n(1)}|$. Then \eqref{eqn:rn} follows because $(X_n(1)-X_n(2))$ converges in probability to something a.s.\ positive (by the same reasoning as in Lemma \ref{lem::wedge_boundarypoints_conv}), $\tau_n(1)$ is tight (as explained above), and $\xi_n$ is a Brownian motion run at speed $\sqrt{\kappa_n}$. 
	\end{proof}
	\vspace{0.1cm}
	
	With Lemma \ref{claim::tightness} in hand, let us take a subsequence $(n_k)_{k\in \N}$ such that along this subsequence \eqref{eqn::hugetuple} converges in distribution to a limit \eqb \label{eqn::sslimit} 
	(h,\eta, X,Y,(f^t)_{t\ge 0}, \tau, h',\eta',r).\eqe
	Note that by Lemma \ref{lem::wedge_boundarypoints_conv}, the joint law of this tuple must be such that $[0,X(q)]$ and $[Y(q),0]$ for $q\in \mathcal{Q}$ have critical $\nu_h$-boundary length equal to $q$.
	We further claim the following.
	\begin{lemma} \label{claim::properties_sslimit}  
		The joint law of \eqref{eqn::sslimit} is such that if $\psi_r$ is the scaling map $z\mapsto rz$ and $f':=\psi_r^{-1}\circ f^{\tau}$, then $(h,\eta,f, h',\eta')$ satisfies conditions (a)-(e).
	\end{lemma} 
	
	We first show how to conclude the proof of Proposition \ref{lem::main_criticalzipperconv} using Lemma \ref{claim::properties_sslimit}, and then turn to the proof of the lemma itself. \vspace{0.1cm}
	
	\begin{proof2}{Proposition \ref{prop:joint-conv}} Letting $(h,\eta,f,h',\eta')$ be as in Lemma \ref{claim::properties_sslimit}, it follows from \cite{MMQ18} that $(f,h',\eta')$ must be measurable with respect to $h$ and $\eta'$. Indeed, if $(h,\eta,f_1,h'_1,\eta'_1,f_2,h'_2,\eta'_2)$ is a coupling such that $(h,\eta,f_i,h'_i,\eta'_i)$ satisfies (a)-(e) for $i=1,2$ then it follows from Theorem \ref{thm2} that $f_1\circ f_2^{-1}$ is a conformal automorphism of $\HH$ that fixes $\{0,\infty\}$, and moreover by (a) and (d), that $f_1(h)$ and $f_2(h)$ give the same critical boundary length to the interval $[0,1]$. This implies that $f_1=f_2$ a.s.\, and so $(f_i,h_i,\eta_i)=(f_i, f_i(h),f_i(\eta))$ are equal for $i=1,2$ a.s.\  Since $(h_n,\eta_n)$ converges in probability to $(h,\eta)$ as $n\to \infty$, and $(\psi_n^{-1}\circ f_n^{\tau_n(1)}, h_n^{\tau_n(1)}\circ \psi_n+Q_{\gamma_n}\log|r_n|, \psi_n^{-1}(\eta_n^{\tau_n(1)}))$ is also measurable with respect to $(h_n,\eta_n)$ for each $n$, this implies that the convergence
		\eqbn \Big(h_n,\eta_n,\psi_n^{-1}\circ f_n^{\tau_n(1)}, h_n^{\tau_n(1)}\circ \psi_n+Q_{\gamma_n}\log|r_n|, \psi_n^{-1}(\eta_n^{\tau_n(1)})\Big)  \to (h,\eta, f,h',\eta') \eqen
		along the subsequence $(n_k)_{k\in \N}$ is actually a limit in probability, and by uniqueness, that it holds along the whole sequence $n\to \infty$.
	\end{proof2}
	\\
	
	\begin{proof2}{Lemma \ref{claim::properties_sslimit}} By Lemma \ref{claim::tightness}, properties (a) and (b) are satisfied. Property (c) is satisfied because $h^{\tau_n(1)}\circ \psi_n+Q\log|r_n|$ and $\psi_n^{-1}(\eta^{\tau_n(1)})$ are independent for every $n$.
		
		To show that property (d) is satisfied, let us via Skorokhod embedding assume that we have joint convergence of the whole tuple a.s.\ along the subsequence $(n_k)_{k\in \N}$. Then since a.s.
		\eqbn  r_{n_k}\to r\in (0,\infty);\,\, \tau_{n_k}(1)\to \tau<\infty;\,\,
		(f_{n_k}^t)_{t\ge 0}\to (f^t)_{t\ge 0} \text{ uniformly on compacts of time and space},
		\eqen
		it follows that $(\psi_{n_k}^{-1}\circ f_{n_k}^{\tau_{n_k}(1)})$ converges to $f$ uniformly on compacts of $\HH$ a.s.  From this, because 
		\eqbn \eta_{n_k}\to \eta \text{ and } \psi_{n_k}^{-1}(\eta_{n_k}^{\tau_{n_k}(1)})=(\psi_{n_k}^{-1}\circ f_{n_k}^{\tau_{n_k}(1)})(\eta_{n_k})\to \eta' \text{ a.s.},\eqen 
		we see that with probability one $\eta'=f(\eta)$.

		To verify that $h'=f(h)$ (since $\eta'$ is independent of $h'$ and the Lebesgue measure of the $\eps$-neighborhood of the $\eta'$ restricted to any compact set goes to zero as $\eps\rta 0$) 
		we only need to check that for any test function $\rho$ with compact support in $\HH$, we have
		\[ \left(h'-2\log|(f^{-1})'| \, , \, |f'|^{-2} (\rho \circ f^{-1}) \right) = (h,\rho) 
		\text{ a.s.}, \]
		where, by definition, \eqb
		(h,\rho) = ( f(h)-2\log|(f^{-1})'|,|f'|^{-2}(\rho\circ f^{-1}) ).
		\eqe
		However, since the support of $\rho$ is compact and we have seen above that $(\psi_{n_k}^{-1}\circ f_{n_k}^{\tau_{n_k}(1)})
		\to f$ uniformly on compacts of $\HH$ a.s., the sequence 
		\eqbn 
		(h_{n_k},\rho)= \left(h_{n_k}^{\tau_{n_k}(1)}\circ
		\psi_{n_k} - Q_{\gamma_{n_k}} \log \big(|(\psi_{n_k}^{-1}\circ f^{\tau_{n_k}}(1))^{-1}\big)'|)\, , \, |(
		\psi_{n_k}^{-1}\circ f_n^{\tau_{n_k}(1)})'|^{-2} (\rho \circ (\psi_{n_k}^{-1}\circ f_n^{\tau_{n_k}(1)})^{-1}) \right)
		\eqen
		converges to 
		$(h'-2\log|(f^{-1})'|,|f'|^{-2}\rho\circ f^{-1})$
		a.s.\ as $k\to \infty$. On the other hand, because $h_{n_k}$ converges to $h$, we have $(h_{n_k}, \rho)\to (h,\rho)$ a.s.\ as $k\to \infty$. This implies the result.

		Finally, we need to show property (e). For this, recall the definitions of $\sigma_n$ and $\sigma$ from the definition of Carath\'{e}odory+ convergence, and note that $\sigma_n(X_n(1))=\tau_n(1)$ for each $n$. Since $(f_{n_k},X_{n_k},Y_{n_k},\tau_{n_k}(1))\Rightarrow (f,X,Y,\tau)$ by assumption, and since convergence in the first coordinate is with respect to the Carath\'{e}odory+ topology,\footnote{Recall that this topology requires \emph{uniform} convergence of the functions $\sigma_n$.} we have that $(\sigma_{n_k}(X_{n_k}(1)),\sigma_{n_k}(Y_{n_k}(1)), \tau_{n_k}(1))\Rightarrow (\sigma(X(1)),\sigma(Y(1)), \tau)$ as $k\to \infty$. On the other hand, the left-hand side is actually equal to $(\tau_{n_k}(1),\tau_{n_k}(1),\tau_{n_k}(1))$ for every $k$, and we clearly have $(\tau_{n_k}(1),\tau_{n_k}(1),\tau_{n_k}(1))\Rightarrow (\tau,\tau,\tau)$. Hence it must be the case that $\sigma(X(1))=\tau=\sigma(Y(1))$ with probability one, and since $f^{\sigma(x)}(x)=0$ for every $x$ (by definition of $\sigma$), this implies that 
		\[ f^{\tau}(X(1))=f^\tau(Y(1))=0 \text{ a.s.}\]
		Since $f=\psi_{r}^{-1}\circ f^\tau$, the same holds a.s.\ if $f^\tau$ is replaced with $f$.
		
		For $q<1$ we observe that the sequence $\tau_n(q)$ is also tight in $n$, and so we may pass to a further subsequence along which the tuple formed by appending $\tau_n(q)$ to \eqref{eqn::hugetuple} converges in distribution to $(h,\eta,X,Y,(f^t)_{t\ge 0}, \tau, h^\tau, \eta^\tau, r, \tau')$. Then repeating the same argument as above with $1$ replaced by $q$, we see that $f^{\tau'}(X(q))=f^{\tau'}(Y(q))=0$ a.s.  Moreover, since $\tau_n(q)\le \tau_n(1)$ for every $n$ we have $\tau'\le \tau$ a.s.\. These two facts together (and using that $(f^t)_{t\ge 0}$ is a centred, reverse Loewner flow) imply that $f^\tau(X(q))=f^\tau(Y(q))$ with probability one. Again, this still holds a.s.\ if $f^\tau$ is replaced with $f$.
	\end{proof2}
	
	\begin{remark0}\label{rmk:stationary_coupling} Observe that if $(h,\eta,f,h',\eta')$ are as in Proposition \ref{prop:joint-conv} then by applying a scaling that puts $h'$ in the last exit parametrisation, we obtain the map from the statement of Theorem \ref{thm:criticalzipper} with $t=1$.
	\end{remark0}

	\label{sec:criticalzipperapprox}

	\section{Proof of main results}\label{sec:mainproofs}
	
	In this section we conclude the proof of Theorems \ref{thm1} and \ref{thm:criticalzipper} by combining results of Sections \ref{sec:zoom} and \ref{sec:approx}.\vspace{0.1cm}
	
	\begin{proof2}{Theorem \ref{thm:criticalzipper}}
		The theorem follows immediately from Remark \ref{rmk:stationary_coupling}, noting that everything generalises trivially if the special value $1$ in Section \ref{sec:criticalzipperapprox} is replaced with any other $t> 0$.
	\end{proof2} 
	
	\begin{propn} Let $(h,\eta,\cX,\cY)$ be such that $(\HH,h,0,\infty)$ is a $(2,1)$-quantum wedge in the last exit parametrisation, $\eta$ is an independent chordal SLE$_4$ in $\HH$ from 0 to $\infty$, and  $\cX,\cY\in\RR$ are sampled by choosing $\cX$ from the uniform distribution on $[0,1]$ and then letting $\cY<0$ be such that $\nu([\cY,0])=\nu([0,\cX])$. Let $D_{\op{L}}\subset\HH$ (resp., $D_{\op{R}}\subset\HH$) be the domain which is to the left (resp., right) of $\eta$. Then the pair of doubly-marked $2$-quantum surfaces $(D_{\op{L}},h+C, \cY,\infty), (D_{\op{R}},h+C,\cX,\infty)$ converge as $C\to \infty$ to a pair of independent $(2,2)$-quantum wedges.
		\label{prop:zoom-twopoints}
	\end{propn}
	\begin{proof}
		The $R$\emph{-unit circle embedding} is defined just as the unit circle embedding (Definition \ref{def:lastexit}), except that $s\mapsto h_{\rad}(e^{-s})-Q_2 s$ hits $R$ (rather than 0) for the first time at $s=0$. For a given $R>1$ make a change of coordinates $z\mapsto rz$ via \eqref{eqn:coc} (with $r$ random and depending on $R$) such that the field $h$ has the $R$-unit circle embedding. Let $\nu$ denote the corresponding boundary length measure, and let 
		$\wh \cX=r\cX$ and $\wh \cY=r\cY$ be the images of $\cX$ and $\cY$ respectively under the change of coordinates.
		Since $\nu([-1,0])$ and $\nu([0,1])$ converge in law to $\infty$ as $R\rta\infty$, we see that with probability converging to 1 as $R\rta\infty$, we have $\wh \cX\in(0,1)$ and $\wh \cY\in(-1,0)$. 
		
		Notice that $h$ restricted to the unit semi-disk $\D^+\subset\HH$ has the law of a free boundary GFF in $\D^+$ plus $z\mapsto -\log |z|$, with additive constant chosen such that the field restricted to the unit semi-circle has average $R$. Let $\cF$ denote the $\sigma$-algebra generated by $h$ restricted to the parts of the imaginary axis and the unit circle that are contained in $\HH$. Let $U_{\op{R}}$ (resp., $U_{\op{L}}$) denote the unit disc restricted to the first (resp., second) quadrant. Then $h|_{U_{\op{R}}}$ and $h|_{U_{\op{L}}}$ are independent conditioned on $\cF$, and $h|_{U_{\op{R}}}$ (resp., $h|_{U_{\op{L}}}$) has the law of a mixed boundary GFF with continuous Dirichlet boundary conditions on $\partial U_{\op{R}}\cap\HH$ (resp., $\partial U_{\op{L}}\cap\HH$) and free boundary conditions on $\partial U_{\op{R}}\cap\RR$ (resp., $\partial U_{\op{L}}\cap\RR$). The proposition now follows from Proposition \ref{prop:zoom} and Lemma \ref{prop:wedge-transform} applied with $U_{\op{L}},U_{\op{R}},D_{\op{L}}$, and $D_{\op{R}}$, when we condition on the $\sigma$-algebra generated by $\cF$ in addition to $\nu([-1,\cY])$, $\nu([\cY,0])$, $\nu([0,\cX])$, and $\nu([\cX,1])$.
	\end{proof}
	
	\begin{proof2}{Theorem \ref{thm1}}
		Consider the ``critical zipper'' $(h_t,\eta_t)_{t\in \R}$ of Theorem \ref{thm:criticalzipper} and 
		let $\cX$ and $\cY$ be as in Proposition \ref{prop:zoom-twopoints} with respect to $(h_0,\eta_0)$. Define $s=\nu([ 0,\cX])$ and ``zip up" by time $\tau(s)=\tau$, i.e., consider $(h_\tau, \eta_{\tau})$, 
		and add a constant $C$ to the field. By Lemma \ref{prop:zoom-twopoints}, as $C\rta\infty$ the surfaces to the left and to the right of $\eta_\tau$ (both with marked points at $0$ and $\infty$) converge to independent $(2,2)$-quantum wedges. Furthermore, the law of the pair $(h_\tau,\eta_\tau)$ is that of a (2,1)-quantum wedge, which is invariant under adding any constant $C$ to the field. Thus, taking a limit as $C\to \infty$ proves the first statement of the theorem. The statement concerning boundary lengths follows directly from Theorem \ref{thm:criticalzipper}. 
	\end{proof2}

	\bibliographystyle{abbrv}
	\bibliography{EP_bibliography}
\end{document}